\documentclass[12pt,a4paper]{article}

\usepackage[utf8]{inputenc}
\usepackage{amsfonts}

\usepackage{amsmath,amsthm,amsfonts,amssymb,amscd,mathtools, bigints, url, bm}

\usepackage{graphicx}
\usepackage{color}
\usepackage{fancyhdr}
\usepackage{subcaption}
\usepackage[shortlabels]{enumitem}
\setlist[enumerate]{nosep}
\usepackage{changepage}

\usepackage{geometry}
\usepackage{a4wide}
\usepackage{authblk}

\author[a,b]{Dániel Garamvölgyi}
\author[b,c]{Tibor Jordán}
\author[b,c]{Csaba Király}

\title{\bf Count and cofactor matroids of highly connected graphs}

\date{ }

\affil[a]{\footnotesize HUN-REN Alfréd Rényi Institute of Mathematics, Reáltanoda utca 13-15, Budapest, 1053, Hungary}
\affil[b]{\footnotesize HUN-REN–ELTE Egerváry Research Group on Combinatorial Optimization, Pázmány Péter sétány 1/C, Budapest, 1117, Hungary}
\affil[c]{\footnotesize Department of Operations Research, ELTE Eötvös Loránd University, Pázmány Péter sétány 1/C, Budapest, 1117, Hungary}
\affil[ ]{\footnotesize \textit{E-mail addresses:} {\tt \{daniel.garamvolgyi,tibor.jordan,csaba.kiraly\}@ttk.elte.hu}}

\newcommand{\cM}{{\cal M}}

\renewcommand{\rho}{\varrho_D}

\newtheorem{theorem}{Theorem}[section]

\newtheorem{lemma}[theorem]{Lemma}
\newtheorem*{lemma*}{Lemma}
\newtheorem*{conjecture*}{Conjecture}
\newtheorem*{lemma''*}{``Lemma''}
\newtheorem{claim}[theorem]{Claim}
\newtheorem*{claim*}{Claim}
\newtheorem{corollary}[theorem]{Corollary}

\newtheorem{conjecture}[theorem]{Conjecture}

\def\X{{\cal X}}

\def\klt{$(k,\ell)$-tight}

\def\klr{$(k,\ell)$-redundant}
\def\klri{$(k,\ell)$-rigid}

\DeclareMathOperator{\val}{val}

\newcommand{\cofactor}{\mathcal{C}}
\newcommand{\cofactort}{\mathcal{C}^t}

\newcommand{\st}{\partial}
\newcommand{\countmatroid}{\mathcal{M}_{k,\ell}}
\newcommand{\bicircular}{\mathcal{M}_{1,0}}

\begin{document}
\maketitle

\setlist[enumerate]{noitemsep,topsep=2pt,label=\textit{(\alph*)}}
\setlist[itemize]{noitemsep,nolistsep}
 
\begin{abstract}
We consider two types of matroids defined on the edge set of a graph $G$:
count matroids ${\cal M}_{k,\ell}(G)$, in which independence is defined by a sparsity count
involving the parameters $k$ and $\ell$, and the 
$C_2^1$-cofactor matroid $\mathcal{C}(G)$, in which independence is defined by
linear independence in the cofactor matrix of $G$.
We show, for each pair $(k,\ell)$, that if $G$ is sufficiently highly connected, then 
$G-e$ has maximum rank for all $e\in E(G)$, and the matroid
${\cal M}_{k,\ell}(G)$ is connected. These results unify and extend several previous results, including theorems 
of Nash-Williams and Tutte ($k=\ell=1$), and Lov\'asz and Yemini ($k=2, \ell=3$). 
We also prove that if $G$ is highly connected, then the vertical connectivity of $\mathcal{C}(G)$ is also high.

We use these results 
to generalize
Whitney's celebrated result on the graphic matroid of $G$ (which corresponds to ${\cal M}_{1,1}(G)$)
to all count matroids and to the
$C_2^1$-cofactor matroid: if $G$ is highly connected, depending on $k$ and $\ell$, then the count matroid
${\cal M}_{k,\ell}(G)$ uniquely determines $G$; and similarly, if $G$ is $14$-connected, then its $C_2^1$-cofactor matroid $\mathcal{C}(G)$ uniquely determines $G$.
We also derive similar results for the $t$-fold union of the $C_2^1$-cofactor matroid, and use them to prove that every $24$-connected graph has a 
spanning tree $T$ for which $G-E(T)$ is
$3$-connected, which verifies a case of a conjecture of Kriesell. 
\end{abstract}

\section{Introduction}
\label{intro}

Let $G$ be a graph and $\mathcal{M}(G)$ a matroid on the edge set of $G$. Given that $G$ has sufficiently high edge or vertex-connectivity, what properties of $\mathcal{M}(G)$ can we deduce?
There are a number of classical theorems in this vein. For example, we have
\begin{itemize}
    \item Whitney's theorem\footnote{Whitney actually gave a more general result characterizing pairs of graphs with isomorphic graphic matroids. Nonetheless, by ``Whitney's theorem'' we shall refer to this corollary throughout the paper.} \cite{whitney}, which says that if $G$ is $3$-connected, then it is uniquely determined by its graphic
		matroid, in the sense that if $H$ is a graph without isolated vertices such that 
		the graphic matroids of $G$ and $H$ are isomorphic, then $G$ and $H$ are isomorphic as well;
    \item a consequence of a theorem of Nash-Williams and Tutte \cite{nash-williams, tutte}, saying that if $G$ is $2k$-edge-connected, then $G$ contains $k$ edge-disjoint spanning trees, or equivalently, the rank 
    of the $k$-fold union of the graphic matroid of $G$
		is $k|V(G)| - k$;
    \item the theorem of Lovász and Yemini \cite{LY}, stating that if $G$ is $6$-connected, then it is redundantly rigid in $\mathbb{R}^2$, which means that for every edge $e$ of $G$, the rank 
    of the generic $2$-dimensional rigidity matroid of $G-e$ is $2|V(G)| - 3$.
\end{itemize}

In all of these examples, the underlying matroids turn out to be \emph{count matroids}, also known as \emph{sparsity matroids}. These matroids are parameterized by two integers $k$ and $\ell$, 
with $k$ positive and $\ell \leq 2k-1$, 
and a set of edges is independent in the $(k,\ell)$-count matroid if the graph induced by them is $(k,\ell)$-sparse, that is, every subset $X$ of vertices induces at most $k|X| - \ell$ edges in the graph. (See the next section for precise definitions.) 

Extensions of the above theorems have been obtained previously for some values of $k$ and $\ell$. For example, an analogue of Whitney's theorem was given for the $(2,3)$-count matroid (that is, the generic $2$-dimensional rigidity matroid) in \cite[Theorem 2.4]{JK}, while the Lovász-Yemini theorem has been generalized to the $(k,2k-1)$-count matroids and $(2k,3k)$-count matroids for each positive integer $k$; see \cite[Theorem 6.2]{JJsparse} and \cite[Theorem 3.1]{J2conn}, respectively.  

We prove generalizations of all of the above theorems, for all count matroids. In particular, we prove that
if a graph is $(\max\{2k,2\ell\} + 1)$-connected, then it is uniquely determined by its $(k,\ell)$-count matroid, when
$k$ and $\ell$ are positive and $\ell \leq 2k-1$
(Corollary \ref{corollary:countwhitney}). The key difficulty is in showing that an analogue of the Lovász-Yemini theorem holds for every $k,\ell$ with $2 \leq k < \ell \leq 2k-1$ (Theorem \ref{LYl>k}). As with all previous generalizations of the Lovász--Yemini-theorem, our method of proof is similar to the original proof of Lovász and Yemini, albeit significantly more involved.

Another example of matroids on graphs is given by the family of generic $d$-dimensional rigidity matroids, denoted by $\mathcal{R}_d(G)$. As we noted above, for $d = 2$ these are the $(2,3)$-count matroids, while for $d = 1$ they coincide with the $(1,1)$-count matroids. For $d \geq 3$, however, generic $d$-dimensional rigidity matroids are not defined by a sparsity condition, and their properties are much less understood than that of count matroids. In particular, finding a combinatorial characterization for the rank functions of  generic $3$-dimensional rigidity matroids is a major open question. 

Brigitte and Herman Servatius \cite[Problem 17]{servatius} asked whether there is a (smallest) constant $k_d$ such that $G$ is uniquely determined by ${\cal R}_d(G)$, provided that ${\cal R}_d(G)$ is $k_d$-connected. The only known cases of this problem are when $d = 1$, where Whitney's theorem gives an affirmative answer, and when $d = 2$, which was answered positively in \cite{JK}. 
We answer the analogous question in the case of 
the so-called $C_2^1$-cofactor matroid $\cofactor(G)$ of a graph $G$.
This matroid bears a strong resemblance to the generic $3$-dimensional rigidity matroid of a graph, leading Whiteley to conjecture that $\cofactor(G) = \mathcal{R}_3(G)$ for every graph $G$, see e.g.,\ \cite[Page 61]{Whlong}. In a recent paper, Clinch, Jackson and Tanigawa \cite{CJT} gave an NP $\cap$ co-NP characterization for the rank function of $\cofactor(G)$. They also showed that every $12$-connected graph is ``$\cofactor$-rigid,'' an analogue of the Lovász-Yemini theorem for $\cofactor(G)$. We use their characterization to show that if $G$ is $14$-connected or if $\cofactor(G)$ is (vertically) $33$-connected, then $G$ is uniquely determined by $\cofactor(G)$ (Theorems \ref{theorem:whitneycof} and \ref{theorem:servatius}). 

In fact, instead of $\cofactor(G)$ we work with its $t$-fold union $\cofactort(G)$. This approach also lets us show that every $12t$-connected graph contains $t$ edge-disjoint $3$-connected spanning subgraphs (Theorem \ref{theorem:kriesell}). It also follows that if a graph $G$ is $24$-connected, then $G$ contains a spanning tree $T$ for which $G - E(T)$ is $3$-connected, which proves the $k=3$ case of a conjecture of Kriesell \cite{kriesell}.  

The rest of the paper is laid out as follows. In Section \ref{section:preliminaries} we give the definitions and facts related to count and cofactor matroids that we shall need. In Sections \ref{section:countmatroids} and \ref{section:cofactormatroids} we consider the relation between the vertex-connectivity of a graph and the vertical connectivity (and other structural properties) of its count and cofactor matroids. In particular, we prove variants of the Lovász-Yemini theorem, as well as related basis packing theorems. Finally, in
Section \ref{section:applications} we apply these results to prove analogues of Whitney's theorem and the special case of the conjecture of Kriesell mentioned above.

\section{Preliminaries}\label{section:preliminaries}

Unless otherwise noted, we consider graphs without loops and isolated vertices, but possibly with parallel edges.
For a graph $G = (V,E)$ and a vertex $v \in V$, $d_G(v)$ denotes the degree of $v$ in $G$, while $\st_G(v)$ denotes the set of edges incident to $v$ in $G$. For a set $X \subseteq V$ of vertices, we let $G[X]$ denote the subgraph of $G$ induced by $X$. For a set of edges $F \subseteq E$, $V(F)$ denotes the set of vertices of the graph induced by $F$.

We assume that the reader is familiar with the basic definitions and results of matroid theory. We refer the reader to \cite{frank,oxley} for more details.

\subsection{Union and vertical connectivity of matroids}

Let $\mathcal{M}_i = (E,\mathcal{I}_i), i \in \{1,\ldots,t\}$ be a collection of matroids 
on a common ground set $E$, where $\mathcal{I}_i$ is the family of independent sets in matroid $\mathcal{M}_i$. 
The \emph{union} of $\mathcal{M}_1,\ldots,\mathcal{M}_t$ is the matroid $\mathcal{M} = (E,\mathcal{I})$ whose independent sets are defined by
\[\mathcal{I} = \{I_1 \cup \ldots \cup I_t : I_1 \in \mathcal{I}_1, \ldots, I_t \in \mathcal{I}_t\}.\]

Nash-Williams \cite{nash-williams-rank} and Edmonds \cite{edmonds-rank} gave the following characterization of the rank function $r$ of $\mathcal{M}$. Let $r_i$ denote the rank function of $\mathcal{M}_i$ for $i \in \{1,\ldots,t\}$. Then for all $E' \subseteq E$
we have
\begin{equation}
\label{union}
r(E') = \min_{F \subseteq E'}\bigl(|F| + \sum_{i=1}^t r_i(E'-F)\bigr).
\end{equation}
Let ${\cal M} = (E,r)$ be a matroid with rank function $r$ 
and let $k$ be a positive integer.
We say that a bipartition $(E_1,E_2)$ of $E$ is a {\it vertical $k$-separation} of $\mathcal{M}$ if $r(E_1), r(E_2) \geq k$ and \[r(E_1) + r(E_2) \leq r(E) + k -1\] holds. In this case $r(E_i) < r(E)$ for $i=1,2$.
The \emph{vertical connectivity}  of $\mathcal{M}$
is defined to be the smallest integer $k$ for which 
${\cal M}$ has a vertical $k$-separation. If ${\cal M}$ has no vertical
separations at all, then we define its vertical connectivity to be $r(E)$.
%
%
We say that ${\cal M}$ is
{\it vertically $k$-connected} if its vertical connectivity is at least $k$.

An element $e\in E$ is a {\it bridge} in ${\cal M}$ if $r(E-\{e\})=r(E)-1$ holds.
We say that a matroid ${\cal M}$ on ground set $E$ is {\it connected} if $|E|\geq 2$ and ${\cal M}$ is loopless and
vertically $2$-connected. In particular, a connected matroid cannot contain any bridges, since if $f \in E$ is a bridge, then $(\{f\}, E-\{f\})$ forms a vertical $1$-separation. 
The maximal connected submatroids
of a matroid ${\cal M}$
are pairwise disjoint and the sum of their ranks is
equal to the rank of ${\cal M}$. 
They are called the {\it components} of ${\cal M}$. We say that a component is \emph{trivial} if it has a single element (which is necessarily a bridge), and \emph{nontrivial} otherwise.

\subsection{Count matroids}

Let $G=(V,E)$ be a graph and let $k$ and $\ell$ be two integers with $k\geq 1$ and $\ell \leq 2k - 1$.
The {\it $(k,\ell)$-count matroid} of $G$ is the matroid ${\cal M}_{k,\ell}(G)=(E,{\cal I}_{k,\ell})$ on the edge set of $G$ in which the family of independent sets is defined by the sparsity condition
\[
\label{indep}
{\cal I}_{k,\ell}=\{ I\subseteq E: |I'|\leq k|V(I')|-\ell\ \hbox{for all}\ \varnothing\not= I'\subseteq I \}.
\]
It is known that ${\cal M}_{k,\ell}(G)$ is indeed a matroid \cite{lorea, Whlong}, whose rank function $r_{k,\ell}$ is given by 
\[r_{k,\ell}(E')= \min \{ |F| + \sum_{Y\in {\cal Y}} (k|V(Y)|-\ell) : F\subseteq E',\  {\cal Y}\ \hbox{is a partition of}\ E'-F \},\]
for $E' \subseteq E$; see \cite[Sections 13.4, 13.5]{frank}.

We can give an alternative formula for $r_{k,\ell}$ using vertex partitions instead of edge partitions.
A {\it cover} of $G = (V,E)$ is a collection ${\cal X}=\{ X_1,X_2,\dots X_t\}$ of subsets of
$V$ of size at least two for which every edge in $E$ is induced by some $X_i$.
When $G$ is clear from the context we shall also say that ${\cal X}$ is a cover of an edge set $J\subseteq E$ to mean that 
it is a cover of the subgraph 
of $G$ induced by $J$.
The cover is said to be {\it $k$-thin}, for some integer $k$, if $|X_i\cap X_j|\leq k$ for all
$1\leq i<j \leq t$.
For a cover
${\cal X}=\{X_1,X_2,...,X_t\}$, we define its \emph{value} (with respect to $k$ and $\ell$) to be
\[\val_{k,\ell}({\cal X})=\sum_{i=1}^t (k|X_i|-\ell).\]
The proof of the following result can be found in the Appendix.

\begin{theorem}
\label{rankthm}
The rank of a set $E' \subseteq E$ of edges in ${\cal M}_{k,\ell}(G)$
is given by
\[
\label{rank2}
r_{k,\ell}(E')= \min \{ |F| + \val_{k,\ell}({\cal X})\},
\]
where the minimum is taken over all subsets $F\subseteq E'$ and all $1$-thin covers ${\cal X}$
of $(V,E'-F)$.
Furthermore, if $0<\ell \leq k,$ then the minimum is attained on a $0$-thin cover of $(V,E'-F)$.
If $\ell \leq 0,$ then 
the minimum is attained on ${\cal X}=\{ V(E'-F) \}$.
\end{theorem}

It is clear from the definitions that for any graph $G = (V,E)$ on at least two vertices, the rank of $\countmatroid(G)$ is at most $k|V| - \ell$. 
For convenience, we introduce the following notions. We say that $G$ is
\begin{itemize}
    \item {\it $(k,\ell)$-rigid} if $r_{k,\ell}(E)=k|V|-\ell$;
    \item \emph{$(k,\ell)$-redundant} if $G-e$ is $(k,\ell)$-rigid for all $e \in E$;
    \item \emph{$(k,\ell)$-sparse} if $r_{k,\ell}(E) = |E|$, or equivalently, if for any set of vertices $X \subseteq V$ of size at least two, the number of edges in $G[X]$ is at most $k|X| - \ell$;
    \item \emph{$(k,\ell)$-tight} if it is both $(k,\ell)$-rigid and $(k,\ell)$-sparse, or equivalently, if it is $(k,\ell)$-rigid and $|E| = k|V| - \ell$.
\end{itemize}
We say that $G$ is an \emph{$\mathcal{M}_{k,\ell}$-circuit} if $E$ is a circuit in ${\cal M}_{k,\ell}(G)$.
Similarly, $G$ is said to be
\emph{$\mathcal{M}_{k,\ell}$-connected} if ${\cal M}_{k,\ell}(G)$ is connected. The subgraphs induced by the components of ${\cal M}_{k,\ell}(G)$ are the {\it ${\cal M}_{k,\ell}$-components} of $G$. 

We record the following facts about the behaviour of count matroids under vertex and edge additions, which are straightforward to deduce from the definitions.

\begin{lemma} 
\label{sf}
Let $k$ and $\ell$ be integers with $k \geq 1$ and $\ell \leq 2k-1$, and let $G$ be a graph.
\begin{enumerate}
    \item Let $G'$ be obtained from $G$ by the addition of a vertex incident to $k$ edges in such a way that we add no more than $2k-\ell$ parallel edges between any two vertices. If $G$ is $(k,\ell)$-sparse (\klt, respectively) then $G'$ is $(k,\ell)$-sparse (\klt, respectively).
    \item Let $\ell' < \ell$ be another integer and let $G'$ be obtained from $G$ by the addition of $\ell - \ell'$ edges (on the same vertex set). If $G$ is \klt, then $G'$ is $(k,\ell')$-tight.
\end{enumerate}
\end{lemma}

The following lemma collects some observations that can be proved by elementary counting arguments.
The complete graph on $n$ vertices is denoted by $K_n$.

\begin{lemma}
\label{simple}
Let $k$ and $\ell$ be integers with $k \geq 1$ and $\ell \leq 2k-1$. Then
\begin{enumerate}
\item if $G = (V,E)$ is a simple graph with $2\leq |V|\leq 2k-1$, then $|E|\leq k|V|-\ell$;
\item if $k+1\leq \ell \leq 2k-1$, then $K_{n}$ is $(k,\ell)$-redundant for all $n\geq 2k$;
\item if $C = (V_C,E_C)$ is an $\countmatroid$-circuit, then $|V_C| = 2$ or $|V_C| > k/(2k-\ell)$.
\end{enumerate}
\end{lemma} 

\begin{proof}
\textit{(a)} $G$ is simple, so
$|E|\leq |V|(|V|-1)/2$. Moreover, from $\ell \leq 2k-1$ we have $k|V| - (2k - 1) \leq k|V| - \ell$. Since for fixed $k$, $|V|(|V|-1)/2\leq k|V|-(2k-1)$ is a quadratic inequality in $|V|$ that holds for $|V| = 2$ and $|V| = 2k-1$, it holds for every $2\leq |V|\leq 2k-1$.   

\textit{(b)} Suppose that $n\geq 2k$.
Let us first consider $K_{2k-1}$. It is $(k,2k-1)$-sparse by \textit{(a)}, and it has $k(2k-1)-(2k-1)$ edges, so it is $(k,2k-1)$-tight. Adding a new vertex of degree $k + (2k-1 - \ell) \leq 2k-2$ results in a \klt\ graph on $2k$ vertices by Lemma \ref{sf}(b), since this is the same as adding a new vertex of degree $k$, which results in a $(k,2k-1)$-tight graph, and then adding $2k-1 - \ell$ additional edges. Note that it is a proper subgraph of $K_{2k}$. By adding 
$n-2k$ additional vertices of degree $k$, we obtain a \klt\ spanning (proper) subgraph of $K_{n}$,
c.f.\ Lemma \ref{sf}(a). 
By symmetry, this shows that $K_{n} - e$ is \klri\ for any edge $e \in E(K_{n})$, and consequently $K_{n}$ is \klr.

\textit{(c)} If $|V_C| \leq 2$, then we are done (recall that we do not allow loops in our graphs), so we may suppose that $|V_C| \geq 3$. If $k/(2k-\ell) \leq 2$, then, again, we are done, so let us suppose that $k/(2k-\ell) > 2$, which is equivalent to $\ell > \frac{3}{2}k$. First, observe that a pair of vertices with $2k-\ell + 1$ parallel edges between them form an $\countmatroid$-circuit. It follows that between any pair of vertices in $C$ there are at most $2k-\ell$ parallel edges. Thus, it suffices to show that the graph $(2k-\ell)K_{V_C}$, consisting of the vertex set $V_C$ and $2k-\ell$ parallel edges between each pair of vertices in $V_C$, is $(k,\ell)$-sparse whenever $|V_C| \leq k/(2k-\ell)$. This follows from a similar calculation as in part \textit{(a)}, as follows. Let $(V_0,E_0)$ be a subgraph of $(2k-\ell)K_{V_C}$. We have $|E_0| \leq (2k-\ell)|V_0|(|V_0|-1)/2$, so it suffices to prove \[(2k-\ell)\frac{|V_0|(|V_0|-1)}{2} \leq k|V_0| - \ell.\] This is a quadratic inequality in $|V_0|$, so it is enough to show that it holds for $|V_0| = 2$ and $|V_0| = k/(2k-\ell)$. The former is immediate. For the latter, observe that after substitution, rearranging, and multiplying by two we obtain the inequality
\[(2\ell-k)(2k-\ell) \leq k^2.\]Since the mapping \[f:x \mapsto (2x - k)(2k - x)\] is a quadratic function with $f(\frac{3}{2}k) = k^2$ and $f'(x) < 0$ for all $x > \frac{5}{4}k$, $f(\ell) \leq k^2$ holds whenever $\ell \geq \frac{3}{2}k$, as desired.
\end{proof}

Finally, the following lemma captures an important property of count matroids. Since we could not find a proof in the literature, we provide one. 

\begin{lemma}
\label{bases}
Let $H$ be a nontrivial
${\cal M}_{k,\ell}$-component of a graph $G$.
Then $H$ is an induced subgraph of $G$ and $H$ is \klr.
\end{lemma}
\begin{proof}
We start by showing that any $\countmatroid$-circuit $C = (V_C,E_C)$ is \klri. Indeed, since $C$ is not $(k,\ell)$-sparse, there must be a subset of vertices $X \subseteq V_C$ of size at least two such that $C[X]$ has at least $k|X|-\ell + 1$ edges. Deleting any edge of $C$ results in a $(k,\ell)$-sparse graph, so we must have $X = V_C$ and $|E_C| = k|V_C| - \ell + 1$. Since $r_{k,\ell}(C) = |E_C| - 1$, $C$ is \klri, as claimed. Moreover, $\countmatroid(C)$ does not contain bridges, so $C$ is, in fact, \klr.

Next, we show that if a graph $G_0$ with at least two edges is $\countmatroid$-connected, then adding any edge $e$ with end vertices $u$ and $v$ to $G_0$ (possibly parallel to an existing edge of $G_0$) also results in an $\countmatroid$-connected graph; in other words, $e$ is not a bridge of $\countmatroid(G_0+e)$. Indeed, let $f,f'$ be (not necessarily distinct) edges of $G_0$ incident to $u$ and $v$, respectively. Since $G_0$ is $\countmatroid$-connected, we can find an $\countmatroid$-circuit $C$ in $G_0$ that contains $f$ and $f'$. In particular, $C$ spans $u$ and $v$. Now $C$ is \klri, so we must have $r_{k,\ell}(C) = r_{k,\ell}(C + e)$. It follows that $e$ is contained in an $\countmatroid$-circuit $C_0$ of $C + e$. Since $C_0$ is also an $\countmatroid$-circuit of $G_0 + e$, $e$ is not a bridge in $\countmatroid(G_0+e)$.

Now if $H = (V_H,E_H)$ is a nontrivial $\countmatroid$-component of a graph $G$, then the above argument shows that $G[V_H]$ is $\countmatroid$-connected. By the maximality of $H$, we must have $H = G[V_H]$, so $H$ is an induced subgraph of $G$. We would like to show that it is also \klr. Again, it is enough to show that $H$ is \klri, since combined with $\countmatroid$-connectivity this implies that $H$ is \klr. 
Moreover, it follows from the argument used in the first part of the proof that adding edges to $H$ does not increase its rank. Thus, it suffices to show that there is a \klri\ supergraph of $H$ on the vertex set $V_H$, which, in turn, is equivalent to showing that there exists a \klt\ graph on $V_H$. 

Now if $|V_H| = 2$, then $2k-\ell$ parallel edges between the two vertices form a \klt\ graph, and we are done. If $|V_H| \geq 3$, then $H$ contains an $\countmatroid$-circuit $C=(V_C,E_C)$ with $|V_C| \geq 3$. Indeed, $H$ has two edges that have different end vertices, and we can take $C$ to be an $\countmatroid$-circuit containing such a pair of edges. It follows from Lemma \ref{simple}(c) that $|V_C| > k/(2k-\ell)$. By the first part of the proof, $C$ is \klri, and thus it contains a spanning \klt\ subgraph $T$. Now we can use $T$ and Lemma \ref{sf}(a) to construct, by vertex additions, a \klt\ graph on $V_H$.
\end{proof}

\subsection{Cofactor matroids}

Let $G$ be a simple graph and $s$ a non-negative integer. The \emph{$C_s^{s-1}$-cofactor matroid} of $G$ is a certain matroid ${\cal C}_{s}^{s-1}(G)$ defined on the edge set of $G$. 
Whiteley \cite{Whlong} proved that ${\cal C}_{d-1}^{d-2}(G)={\cal R}_d(G)$ for $d=1,2$.
We shall focus on the $d=3$ case (that is, the $C_2^1$-cofactor matroid) 
where, recently,  Clinch, Jackson, and Tanigawa \cite{CJT} gave a combinatorial characterization for the rank function of ${\cal C}_{2}^{1}(G)$.
As the definition of this matroid is rather technical and not directly relevant for our purposes, 
we shall take this characterization as a definition, and direct the interested reader to \cite{Whlong} for a detailed treatment of cofactor matroids.
Throughout the rest of the paper we shall adopt the simpler notation $\cofactor(G)$ for $\mathcal{C}_2^1(G)$.

Let $G=(V,E)$ be a graph and let ${\cal X}$ be a $2$-thin cover of $G$. A
{\it hinge} of ${\cal X}$ is a pair of vertices $\{x,y\}$ with $X_i\cap X_j=\{x,y\}$
for two distinct $X_i,X_j\in {\cal X}$. We use $H({\cal X})$ to denote the set of all hinges
of ${\cal X}$. The {\it degree} $\deg_{\cal X}(h)$ of a hinge $h$ of ${\cal X}$ is the
number of sets in ${\cal X}$ which contain $h$. 
The family ${\cal X}$ is called {\it $k$-shellable} if its elements can be ordered
as a sequence $(X_1,X_2,\dots ,X_m)$ so that, for all $2\leq i\leq m$, we have
$|X_i\cap \bigcup_{j=1}^{i-1} X_j|\leq k$.

\begin{theorem} \cite[Theorem 6.1]{CJT}
\label{cofactorrank}
Let $G = (V,E)$ be a simple graph and let $r$ denote the rank function of $\mathcal{C}(G)$. Then for each $E' \subseteq E$, we have
 \[
r(E') = \min \{ |F| + \sum_{X\in {\cal X}} (3|X|-6) - \sum_{h\in H(\X)} (\deg_{\X}(h)-1)\},
\]
where the minimum is taken over all subsets $F\subseteq E'$ and all $4$-shellable $2$-thin covers ${\cal X}$
of $(V,E'-F)$ with sets of size at least five.
\end{theorem}

We shall need the following generalization of Theorem \ref{cofactorrank} to the union of $t$ copies of $\mathcal{C}(G)$,
which we denote by $\cofactort(G)$.

\begin{theorem}\label{theorem:cofactortrank}
Let $G=(V,E)$ be a simple graph and let $r_t$ denote the rank function of $\mathcal{C}^t(G)$. Then for each $E' \subseteq E$, we have
\begin{equation}
\label{trank}
r_t(E')= \min \{ |F|+ t\sum_{X\in \X} (3|X|-6) - t\sum_{h\in H(\X)} (\deg_{\X}(h)-1),
\end{equation}
where the minimum is taken over all $F\subseteq E'$ and all $4$-shellable $2$-thin covers ${\cal X}$
of $(V,E'-F)$ with sets of size at least five.
\end{theorem}

\begin{proof}
By using the rank formula (\ref{union}) of the union of matroids, we obtain
\[r_t(E')= \min_{F\subseteq E'} \{ |F|+ t \cdot r_1(E'-F) \}.\] Using Theorem \ref{cofactorrank} we can rewrite this as 
\[r_t(E')= \min_{F\subseteq E'} \{ |F| + t \bigl(\min \{|F'|+ \sum_{X\in \X} (3|X|-6) - \sum_{h\in H(\X)} (\deg_{\X}(h)-1)\bigr),\]
where the second minimum is taken over all $F'\subseteq E'-F$ and $4$-shellable 2-thin covers ${\cal X}$ of
$(V,E'-(F\cup F'))$.
Since $t\geq 1$, replacing 
$F$ by $F\cup F'$ and $F'$ by the empty set in a minimizing triple $F,F',{\cal X}$ does not increase the right hand side.
Therefore we can simplify this formula and
deduce that (\ref{trank}) holds.
\end{proof}

We can observe that $r_t(\cofactort(G)) \leq 3t|V| - 6t$ holds for any simple graph $G$ on at least $5$ vertices by applying Theorem \ref{theorem:cofactortrank} with $F = \varnothing$ and $\mathcal{X} = \{V\}$. Also note that if $F,\mathcal{X}$ is a pair for which equality holds in (\ref{trank}), then the members of $F$ are bridges in $\cofactort(G)$. Indeed, for any $f \in F$, by considering $G-f, F-f$ and $\mathcal{X}$, and applying Theorem \ref{theorem:cofactortrank}, we obtain $r_t(G-f) = r_t(G) - 1$.

We shall repeatedly use the following ``vertex addition lemma'' for $\cofactort$. The statement follows immediately from the special case when $t = 1$, which can be found in \cite[Lemma 10.1.5]{Whlong}.
\begin{lemma}\label{lemma:cofactorbridge}
Let $G = (V,E)$ be a simple graph, $v$ a vertex of $G$ and $t$ a positive integer. Then $r_t(E) \geq r_t(E-\st_G(v)) + \min\{3t,d_G(v)\}$. In particular, if $d_G(v) \leq 3t$, then every edge incident to $v$ is a bridge in $\cofactort(G)$.
\end{lemma}

\begin{lemma}\label{lemma:cofactorcompleterank}
Let $n \geq 6t$ be an integer. Then $r_t(K_n) = 3tn - 6t$.
\end{lemma}
\begin{proof}
Since we know that $r_t(K_n) \leq 3tn-6t$, it is sufficient to show that $r_t(K_n) \geq 3tn-6t$ holds. We proceed by induction on $n$. First, let $n = 6t$. We shall show that in this case $K_{n}$ contains $t$ edge-disjoint subgraphs $G_1,\ldots,G_t$ with $r_1(G_i) = 3n - 6$ for $i=1,\ldots,t$. 

Let $V_1,\ldots,V_t$ be a partition of $V(K_n)$ into sets of size $6$ and, initially, 
let $G_i$ be the complete graph on $V_i$. For each pair of indices $i,j$ with $1 \leq i < j \leq t$ we add edges to $G_i$ and $G_j$ as follows. Let $v_1,\ldots,v_6$ and $w_1,\ldots,w_6$ denote the vertices of $V_i$ and $V_j$, respectively. Now we add the edges $v_kw_l, k,l \in \{1,2,3\}$ and the edges $v_kw_l, k,l \in \{4,5,6\}$ to $G_i$, and the edges $v_kw_l, k \in \{1,2,3\}, l \in \{4,5,6\}$ and the edges $v_kw_l, k \in \{4,5,6\}, l \in \{1,2,3\}$ to $G_j$.

In this way, we obtain edge-disjoint spanning subgraphs $G_1,\ldots,G_t$ of $K_n$. Moreover, for each $i \in \{1,\ldots,t\}$, $G_i$ can be obtained from a copy of $K_6$ by adding $n-6$ vertices of degree $3$. Since $r_1(K_4) = 6 = 3 \cdot 4 - 6$ (which can be seen from Theorem \ref{cofactorrank}), Lemma \ref{lemma:cofactorbridge} implies that $r_1(G_i) = 3n - 6$. Since $G_1,\ldots,G_t$ are edge-disjoint, we have $r_t(K_n) \geq \sum_{i=1}^{t}r_1(G_i) = 3tn-6t$.

Now let $n > 6t$. Since $K_n$ can be obtained from $K_{n-1}$ by adding a vertex of degree $n-1 > 3t$, the induction hypothesis and Lemma \ref{lemma:cofactorbridge} imply \[r_t(K_n) \geq 3t(n-1) - 6t + 3t = 3tn -6t,\]as desired.
\end{proof}

\section{Count matroids of highly connected graphs}\label{section:countmatroids}

In this section we give sufficient conditions for a graph to be $(k,\ell)$-rigid. There are three distinct subcases depending on the value of $\ell$ in relation to $k$, the third of which is significantly more difficult than the other two. First, when $\ell \leq 0$, then we only need a bound on the minimum degree to ensure $(k,\ell)$-rigidity, and a slightly higher bound also gives $(k,\ell)$-redundancy. In the $0 < \ell \leq k$ case we show that if a graph is $2k$-edge-connected, then it is $(k,\ell)$-redundant. As we shall see, this follows quickly from (a slight extension of) the result of Tutte and Nash-Williams that $2k$-edge-connected graphs contain $k$ edge-disjoint spanning trees, that is, are $(k,k)$-rigid. Finally, for $k < \ell \leq 2k-1$ we show that $2\ell$-connected graphs are $(k,\ell)$-redundant, an extension of the Lovász--Yemini-theorem. In all three cases we also give conditions that ensure $\countmatroid$-connectivity.


\subsection{The 
{$\ell \leq 0$}
and 
{$0 < \ell \leq k$}
cases}

For the $\ell \leq 0$ case, we shall give an argument using orientations of graphs. We say that an orientation $\vec{G}$ of $G$ is \emph{smooth} if for each vertex $v$, the in-degree and out-degree of $v$ differ by at most $1$. It is well-known (see, for example, \cite[Theorem 1.3.8]{frank}) 
that every graph has a smooth orientation, which can be obtained by adding a perfect matching between the vertices of odd degree, taking an Eulerian orientation of the resulting graph and restricting this orientation to the original edges of the graph.

\begin{theorem}\label{LYl<0}
Let $k$ and $\ell$ be integers such that $k\geq 1$ and $\ell \leq 0$ and let $G = (V,E)$ be a graph. If the degree of each vertex of $G$ is at least $2k-\frac{2\ell}{|V|}$, then $G$ is \klri. Moreover, if each degree is at least $2k-\frac{2\ell-2}{|V|}$, then $G$ is \klr.
\end{theorem}

\begin{proof}
We first show that $G$ has a $(k,0)$-tight spanning subgraph $G_0$. Take a smooth orientation $\vec{G}$ of $G$. Since the degree of each vertex in $G$ is at least $2k$, the in-degree of each vertex in $\vec{G}$ is at least $k$. It follows that we can find a spanning subdigraph $\vec{G_0}$ in which the in-degree of each vertex is $k$. The underlying undirected graph $G_0$ of $\vec{G_0}$ is $(k,0)$-sparse and has $k|V|$ edges, so it is $(k,0)$-tight.

By adding up the degree of each vertex in $G$ we have $2|E| \geq 2k|V| - 2\ell$, and thus $|E - E(G_0)| \geq  - \ell$. Now we can add $-\ell$ edges from $E - E(G_0)$ to $G_0$ to obtain a spanning \klt\ subgraph of $G$, which shows that $G$ is \klri\ (c.f.\ Lemma \ref{sf}(b)).

Furthermore, if the degree of each vertex in $G$ is at least $2k-\frac{2\ell-2}{|V|}$, then after deleting an arbitrary edge $e$ of $G$, each degree remains at least $2k$ and $G-e$ still has at least $k|V|-\ell$ edges. Hence, the previous argument shows that $G-e$ is \klri\ for each edge $e \in E$, and therefore $G$ is \klr\ in this case.
\end{proof}

Next, we consider the case when $0 < \ell \leq k$. We shall need the following theorem which we already mentioned in the Introduction. For a proof, see e.g.\ \cite[Corollary 10.5.2]{frank}.\footnote{Frank gives a slightly weaker statement where no edges of $G$ are deleted, but the same proof works for our statement.}

\begin{theorem} \cite{nash-williams,tutte}
\label{theorem:tutte}
Let $G = (V,E)$ be a $2k$-edge-connected graph. Then for any subset of edges $E'$ of size at most $k$, $G - E'$ contains $k$ edge-disjoint spanning trees. 
\end{theorem}

\begin{theorem}\label{LYl<k}
Let $k$ and $\ell$ be integers such that $0 < \ell \leq k$ and let $G = (V,E)$ be a graph. If $G$ is $2k$-edge-connected, then $G$ is \klr.
\end{theorem}
\begin{proof}
Let $e$ be an edge of $G$ and $E' \subseteq E-e$ a set of edges of size $k - \ell$. Note that $|E' + e| = k - \ell + 1 \leq k$ holds, since $\ell > 0$. Now Theorem \ref{theorem:tutte} and the fact that $G$ is $2k$-edge-connected together imply that $G - E' - e$ contains $k$ edge-disjoint spanning trees, i.e.\ a $(k,k)$-tight spanning subgraph $G_0$. It follows that $G_0 + E'$ is a \klt\ spanning subgraph of $G-e$ by Lemma \ref{sf}(b). This shows that $G-e$ is \klri\ for each edge $e \in E$, and therefore $G$ is \klr.
\end{proof}

We note that Theorems \ref{LYl<0} and \ref{LYl<k} can also be deduced from 
the rank formula of Theorem \ref{rankthm} by relatively simple counting arguments.
On the other hand, the proofs given here are algorithmic.

Now we turn to ${\cal M}_{k,\ell}$-connectivity. We need the following lemma. 

\begin{lemma}\cite[Lemma 3.3]{graug} \label{Mconn} 
Let $G$ be a \klr\ graph.
Then
\begin{enumerate}
    \item if $\ell \leq 0$ and $G$ is connected, then $G$ is ${\cal M}_{k,\ell}$-connected,
    \item if $0< \ell \leq k$ and $G$ is $2$-connected, then $G$ is ${\cal M}_{k,\ell}$-connected.
\end{enumerate}
\end{lemma}
Combining Lemma \ref{Mconn} with Theorems \ref{LYl<0} and \ref{LYl<k} we obtain the following corollaries.
\begin{corollary}\label{MLYl<0}
Let $k$ be a positive integer and $\ell\leq 0$.
If $G=(V,E)$ is connected and
the degree of each vertex of $G$ is at least $2k-(2\ell-2)/|V|$, then $G$ is ${\cal M}_{k,\ell}$-connected.
\end{corollary}

\begin{corollary}\label{MLYl<k}
Let $k$ be a positive integer and $k\geq \ell> 0$. 
If $G=(V,E)$ is $2k$-edge-connected and 2-connected, then $G$ is ${\cal M}_{k,\ell}$-connected.\qed
\end{corollary}

\subsection{The 
{$k < \ell$}
case}

The $k<\ell$ case is more difficult. We shall extend and simplify the previous proofs that solved
the special case $(k,\ell) = (2,3)$ \cite{LY} and, more generally, the case $(k,2k-1)$ \cite{JJsparse}.

The main result of this section is the following theorem.
It shows that sufficiently highly vertex-connected graphs are \klr.
 
\begin{theorem}\label{LYl>k}
Let $k$ and $\ell$ be two positive integers with $2\leq k<\ell\leq 2k-1$. 
Then every $2\ell$-connected graph is
\klr.
\end{theorem}

\begin{proof}
It is enough to prove for simple graphs, since passing to the underlying simple graph preserves vertex-connectivity, and the addition of parallel edges preserves $(k,\ell)$-redundancy. 
For a contradiction, suppose that the statement is false and consider
a counterexample $G=(V,E)$ for which $|V|$ is as small as possible, and with respect to this,
$|E|$ is as large as possible. Hence $G$ is $2\ell$-connected and has an edge $e$ for which $G-e$ is not \klri.
Theorem \ref{rankthm} implies that
there exists a set $F_0 \subseteq E - e$ and a $1$-thin cover ${\cal X}=\{X_1,X_2,\dots, X_t\}$  of $G - F_0 - e$ for which
\begin{equation}
\label{eq1}
k|V|-\ell > |F_0|+ 
\val_{k,\ell}({\cal X})
\end{equation}
Let us choose $e$ and the pair $F_0$, ${\cal X}$ for $G-e$ so that 
$|F_0|+ \val_{k,\ell}({\cal X})$ is as small as possible and with respect to this,
$|F_0|$ is as large as possible. Let $F$ denote $F_0 + e$.
Note that the maximality of $|E|$ implies that $G[X_i]$ is complete for $i=1,\dots,t$.

\begin{claim}\label{Xik} 
$|X_i|\geq 2k$ for all $1\leq i\leq t$.
\end{claim}

\begin{proof}
Lemma \ref{simple}(a) implies that if $|X_i|\leq 2k-1$ then removing $X_i$ from ${\cal X}$ and adding
the edge set induced by $X_i$ in $G-e$ to $F_0$ gives rise to a pair $F_0'$, ${\cal X}'$, where
${\cal X}'$ is a $1$-thin cover of $E-F_0'-e$ and  $|F_0'|+ \val_{k,\ell}({\cal X}') \leq 
|F_0|+ \val_{k,\ell}({\cal X})$. Since $|F_0'| > |F_0|$, this contradicts the choice of $F_0$ and ${\cal X}$.
\end{proof}

\begin{claim} 
\label{c2}
Each vertex $v\in V$ satisfies exactly one 
of the following.
\begin{enumerate}[label=\textit{(\roman*)}]
\item $v$ is contained by at least two sets $X_i,X_j \in {\cal X}$.
\item $v$ is contained by exactly one set $X_i\in {\cal X}$, and $v$ is incident with at least one edge in $F$.
\item No set in ${\cal X}$ contains $v$, and $v$ is incident with at least $2\ell$ edges in $F$.
\end{enumerate}
\end{claim}

\begin{proof}
Since $G$ is $2\ell$-connected, each vertex has degree at least $2\ell$. Thus if no set in ${\cal X}$ contains $v$ then 
every edge incident with $v$ belongs to $F$ and $v$ satisfies \textit{(iii)}. So we may assume that at least one set in ${\cal X}$
contains $v$. Suppose that \textit{(i)} does not hold. Then $v$ is contained by exactly one set, say $X_1\in {\cal X}$.
It remains to show that $v$ is incident with at least one edge in $F$.

Let us assume, for a contradiction, that no edge incident with $v$ is in $F$. It follows that $X_1$ covers all the edges incident with $v$ and hence
it contains all the neighbors of $v$ in $G$, so $|X_1| \geq 2\ell+1$.
Let $G'=G-v$, $X_1'=X_1-v$ and let ${\cal X}'=\{X_1',X_2,\dots,X_t\}$.
Now $F\subseteq E(G')$, $|X_1'|\geq 2\ell$, and ${\cal X}'$ is a $1$-thin cover of $E(G')-F$.
Furthermore, \[r_{k,\ell}(G' - e)\leq |F_0|+\val_{k,\ell}({\cal X}')=|F_0|+\val_{k,\ell}({\cal X})-k< k|V| - \ell -k = k|V(G')|-\ell,\]
which shows that $G'-e$ is not \klri.
The minimal choice of $G$ then implies that $G'$ is not $2\ell$-connected. So either $G-v$ has $2\ell$ vertices and hence $X_1=V$ and $G=K_{2\ell+1}$ hold, or $G-v$ 
has a set $S$ of vertices with $|S|=2\ell-1$ such that $G-v-S$ is disconnected. 
The former case is not possible, since Lemma \ref{simple}(b) shows that $K_{2\ell + 1}$ is \klr.
Let us focus on the latter case.
The $2\ell$-connectivity of $G$ 
implies that $v$ has at least one neighbor in $G$ in 
each connected component of $G'-S$. Thus $X_1$ intersects the vertex set of each connected component of $G-S'$.
But $X_1$ induces a complete subgraph of $G$, a contradiction.
\end{proof}

\begin{claim}
\label{c3}
For each $v\in V$, we have
\[
\frac{d_F(v)}2 + \sum_{i:X_i\ni v} \bigl(k-\frac{\ell}{|X_i|}\bigr) \geq k.
\]
\end{claim}

\begin{proof}
First observe that $|X_i|\geq 2k$ and $\ell \leq 2k-1$ imply that $\frac{\ell}{|X_i|}<1$. 
Thus if $v$ satisfies Claim \ref{c2}\textit{(i)}, then 
$\sum_{i:X_i\ni v} (k-\ell/|X_i|) > 2k-2 \geq k$, and the claim follows.


If $v$ satisfies Claim \ref{c2}\textit{(ii)} then either $d_F(v)\geq 2$, or $d_F(v)=1$.
In the former case we have $\frac{d_F(v)}2 + (k-\frac{\ell}{|X_i|}) > k$, where $X_i\in {\cal X}$ is the set which contains $v$.
In the latter case
$X_i$ must contain all but one of the neighbors of $v$. This implies $|X_i|\geq 2\ell$, and hence $\frac{d_F(v)}2+ k-\frac{\ell}{|X_i|}\geq \frac{1}{2}+ k - \frac{1}{2} = k$. 

Finally, if $v$ is incident with $2\ell$ edges in $F$ then then the first term is at least
$\ell$, which completes the proof by noting that $\ell > k$.
\end{proof}

We can now use
Claim \ref{c3} to obtain
\begin{equation}
\label{eq2}
|F|+ \val_{k,\ell}(\mathcal{X}) =\sum_{v\in V}\Bigl(\frac{d_F(v)}2+\sum_{i:X_i\ni v} \bigl(k-\frac{\ell}{|X_i|}\bigr)\Bigr) \geq k|V|.
\end{equation}
Since $|F| = |F_0| + 1$ and $\ell \geq 3$, the inequalities (\ref{eq1}) and (\ref{eq2}) together give
\[|F| + \val_{k,\ell}(\mathcal{X}) \geq k|V| \geq k|V| - \ell + 3 > |F| - 1 + \val_{k,\ell}(\mathcal{X}) + 3,\]a contradiction.
\end{proof}

The following construction shows that
the bound $2\ell$ is best possible. 
Take the disjoint union of $2\ell+2$ copies of $K_{2\ell-1}$ on vertex sets $V_1,\ldots,V_{2\ell+2}$; for convenience, let $V_{2\ell+3}$ also denote $V_1$. For each index $i$ with $1 \leq i \leq 2\ell+2$, add a matching of size $\ell-1$ between $V_i$ and $V_{i+1}$ in such a way that the edges in different matchings are also pairwise disjoint. This leaves a single vertex $v_i \in V_i$ that is not incident to a matching edge, for each $i = 1,\ldots,2\ell+2$. Finally, add an edge between $v_i$ and $v_{i + \ell + 1}$ for each $i = 1,\ldots,\ell+1$. 
See Figure \ref{figure:LovaszYemini} for the case when $(k,\ell) = (2,3)$. We note that for this case a similar example was given in \cite{LY}.

\begin{figure}[ht]
    \centering
    \includegraphics[width=0.5\linewidth]{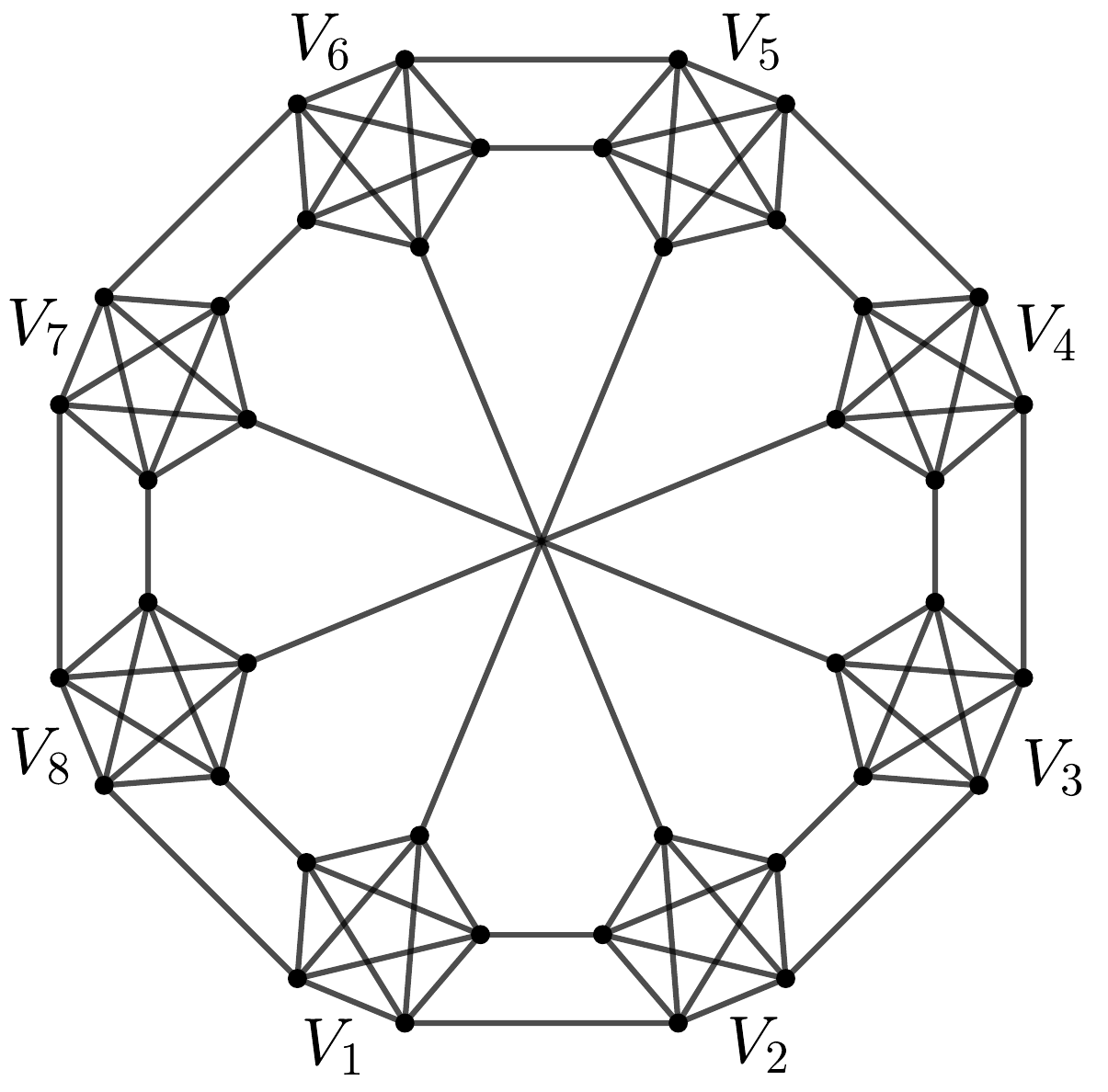}
    \caption{An example of a $5$-connected graph that is not $(2,3)$-rigid.}
    \label{figure:LovaszYemini}
\end{figure}

The resulting graph $G=(V,E)$ is $(2\ell-1)$-connected. 
Let $\mathcal{X} = \{V_1,\ldots,V_{2\ell+2}\}$ and $F = E - \cup_{i=1}^{2\ell+2}E(V_i)$. 
Now $\mathcal{X}$ is a $0$-thin cover of $E-F$, and a straightforward computation shows that
\[|F|+\val_{k,\ell}({\cal X})\leq k|V|-\ell-1,\]
which implies, by Theorem \ref{rankthm}, that $G$ is not \klri\ (and hence not \klr).

We can also prove the following strenghtening of Theorem \ref{LYl>k} using standard methods, see e.g.,\ \cite[Theorem 3.2]{JJ} or \cite[Lemma 3.3]{graug}.

\begin{theorem}\label{MLYl>k}
Let $2\leq k<\ell\leq 2k-1$ be two positive integers and let $G = (V,E)$ be a graph. If $G$ is $2\ell$-connected, then $G$ is ${\cal M}_{k,\ell}$-connected.
\end{theorem}

\begin{proof}
It is enough to prove for simple graphs, since passing to the underlying simple graph preserves vertex-connectivity, while it follows from Lemma \ref{bases} that adding parallel edges preserves the property of being $\countmatroid$-connected. 

Suppose, for a contradiction, that $G$ is not ${\cal M}_{k,\ell}$-connected and let $H_1, \dots, H_q$ be the 
${\cal M}_{k,\ell}$-components of $G$. 
Each $H_i$ is nontrivial, since $G$ is \klr\ by Theorem \ref{LYl>k}.
Let $X_i=V(H_i)- \bigcup_{j \neq i} V(H_j)$ and
%
let $Y_i=V(H_i)-X_i$, $1\leq i\leq q$. We have $|V|=\sum\limits_{i=1}^{q} |X_i|+|\bigcup\limits_{i=1}^{q} Y_i|$ and $\sum\limits_{i=1}^{q} |Y_i| \geq 2|\bigcup\limits_{i=1}^{q} Y_i|$, which gives $|V| \leq \sum\limits_{i=1}^{q}|X_i|+\frac{1}{2}\sum\limits_{i=1}^{q}|Y_i|$.

It follows from Lemma \ref{simple}\textit{(a)} that 
$|V(H_i)|\geq 2k$ for $1\leq i\leq q$.
The $2\ell$-connectivity of $G$ implies that $|Y_i|\geq 2\ell$ or $X_i=\varnothing$ must hold for each $Y_i$; thus in both cases we have
$|Y_i|\geq 2k$ for every $1\leq i\leq q$.

Since $G$ is \klri, it has a spanning \klt\ subgraph $(V,B)$.
%
Let $B_i=B\cap E(H_i)$, 
for $i=1, \dots,q$. Thus $\bigcup\limits_{i=1}^{q} B_i=B$. 
Note that $B_i$ is a base of $H_i$ for $1\leq i\leq q$.
By using the above inequalities and Lemma~\ref{bases}, we obtain
\begin{equation*}
\begin{aligned}
k|V|-\ell &=|\bigcup\limits_{i=1}^{q}B_i|= \sum\limits_{i=1}^{q}|B_i|=\sum\limits_{i=1}^{q} (k|V(H_i)|-\ell)= k \sum\limits_{i=1}^{q}|X_i|+k\sum\limits_{i=1}^{q}|Y_i| -q\ell
\\ &= k\Bigl(\sum\limits_{i=1}^{q}|X_i|+\frac{1}{2}\sum\limits_{i=1}^{q}|Y_i|\Bigr)+ \frac{k}{2}\sum\limits_{i=1}^{q}|Y_i|-q\ell 
\geq k|V| + \frac{k}{2}\sum\limits_{i=1}^{q}|Y_i|-q\ell 
\\ &\geq k|V|+ \frac{k\cdot q\cdot 2k}{2} -q\ell>k|V|, 
\end{aligned}
\end{equation*}
where the last inequality follows from $2\leq k<\ell\leq 2k-1$. 
This contradiction completes the proof.
\end{proof}

We close this section by highlighting the following ``basis packing'' reformulation of Theorem \ref{LYl>k}. The $(k,\ell) = (2,3)$ case of this result can be found in \cite[Theorem 3.1]{J2conn}.

\begin{theorem}\label{theorem:countmatroidbasispacking} Let $k,\ell$ and $t$ be positive integers with $2\leq k<\ell\leq 2k-1$ and let $G$ be a graph. If $G$ is $(2 \ell \cdot t)$-connected, then it contains $t$ edge-disjoint \klri\ spanning subgraphs.
\end{theorem}
\begin{proof}
By Theorem \ref{LYl>k}, $G$ is $(kt,\ell t)$-rigid. The key observation is that $\mathcal{M}_{kt,\ell t}(G)$ is precisely the $t$-fold union of $\mathcal{M}_{k,\ell}$. This follows from Theorem \ref{rankthm} and (\ref{union}); since the proof is simple and analogous to that of 
Theorem \ref{theorem:cofactortrank}, we omit it. It follows that $G$ contains $t$ edge-disjoint $(k,\ell)$-tight (and thus $(k,\ell)$-rigid) spanning subgraphs.
\end{proof}

\section{Cofactor matroids of highly connected graphs}\label{section:cofactormatroids}

In this section we consider properties of $\cofactort(G)$, the $t$-fold union of the generic three-dimensional cofactor matroid of $G$. In particular, we show that if $G = (V,E)$ is $12t$-connected, then $r_t(E) = 3t|V| - 6t$ (i.e.,\ $G$ is ``$\cofactort$-rigid''). We also show that if $G$ is sufficiently highly connected, then $\cofactort(G)$ has high vertical connectivity, and conversely, if $\cofactort(G)$ is sufficiently highly vertically connected, then $G$ has high vertex-connectivity.
Throughout this section we only consider simple graphs.

We say that a bipartition $\{E_1,E_2\}$ of $E$ is \textit{essential} (with respect to $t$) if 
\[
\max \{r_t(E_1), r_t(E_2) \} < 3t|V|-6t.
\]
The following lemma is our main technical result in this section. The proof uses ideas from \cite[Lemma 5.2.]{JK} and \cite[Theorem 7.2]{CJT}.

\begin{lemma}\label{lemmaLYcofact}
Let $t$ be a positive integer and let $G=(V,E)$ be a $12t$-connected graph.
Then for every essential partition
$\{E_1,E_2\}$ of $E$ we have 
\[
r_t(E_1)+r_t(E_2)\geq 3t|V|.
\]
\end{lemma}

\begin{proof}
Let us suppose, for a contradiction, that the lemma does not hold and 
let $G=(V,E)$ be a counterexample for which $|V|$ is as small as possible and, with respect to this, $|E|$ is as large as possible. 
Let $\{E_1,E_2\}$ be an essential partition of $E$ for which
\[r_t(E_1)+r_t(E_2)<3t|V|.\]
%

By Theorem \ref{theorem:cofactortrank}, there exist sets $F_i\subseteq E_i$ and
$4$-shellable $2$-thin covers $\X_i$ of $E_i-F_i$, consisting of sets of size at least five,
with hinge sets $H_i$, for $i=1,2$,
such that
\[r_t(E_i)=|F_i|+\sum_{X\in \X_i} (3t|X|-6t) - t \sum_{h\in H_i} (\deg_{\X_i}(h)-1).\]
Let $\X=\X_1\cup \X_2$ and $F=F_1\cup F_2$. 


It follows from the maximality of $|E|$ that $G[X]$ is complete for every set $X \in \X$. 
This also implies that $F_i=E_i-\bigcup_{X\in\X_i} K_X$, for $i=1,2$, where $K_X$ denotes
the complete graph on vertex set $X$.

\begin{claim}\label{claim:types}
Each vertex $v\in V$ satisfies exactly one of the following.

(i) $v$ is contained by at least two sets $X,Y\in\X$,

(ii) $v$ is contained by exactly one set $X\in \X$, and $v$ is incident with at least one edge in $F$,

(iii) no set in $\X$ contains $v$, and $v$ is incident with at least $12t$ edges in $F$.
\end{claim}

\begin{proof}
Since $G$ is $12t$-connected, we have $d(v)\geq 12t$. Suppose that
no set in $\X$ contains $v$. Then every edge incident with $v$ belongs to $F$
and (iii) holds. If $v$ is contained by some set in $\X$ then either
(i) holds or $v$ is contained by exactly one set, say $X_0\in \X$.
It remains to show that in the latter case
$d_F(v)\geq 1$, where $d_F(v)$ is the number of edges in $F$ incident with $v$.


For a contradiction suppose
that $d_F(v)=0$. By symmetry we may assume that $X_0\in \X_1$. This also implies that there is no edge in $E_2$ 
incident with $v$. Since $\X_1$ covers $E_1$ and no edge in $F_1$ is incident with $v$,
we have $|X_0|\geq 12t+1$ as $X_0$ must contain $v$ and all (at least $12t$) neighbors of $v$ in $G$.

Let us consider the case when $G$ is a complete graph. In this case $X_0=V$ must hold and hence $r_t(E_1)= 3t|V|-6t$ follows by Lemma \ref{lemma:cofactorcompleterank}, 
which 
contradicts the fact that $\{E_1,E_2\}$ is essential. Hence we may assume that $G$ is not complete.
This implies $|V|\geq 12t+2$.

We next show that $\{E_1-\st_G(v),E_2\}$ is an essential partition of the edge set $E-\st_G(v)$
of $G-v$. If this is not the case, then one of two possibilities must hold. The first one is that
\[r_t(E_1-\st_G(v))=3t|V(G-v)|-6t=3t|V|-6t - 3t.\] In this case we can use Lemma \ref{lemma:cofactorbridge} to deduce
that $r_t(E_1)=3t|V|-6t$, a contradiction.
The second possibility is that
\[r_t(E_2-\st_G(v))=r_t(E_2)=3t|V(G-v)|-6t.\]
Since $X_0$ induces a complete graph on at least $12t+1$ vertices in $E_1$, Lemma \ref{lemma:cofactorcompleterank} implies $r_t(E_1)\geq 3t(12t+1)-6t$. It follows that
\[r_t(E_1)+r_t(E_2)\geq 36t^2-3t+3t|V|-9t>3t|V|,\] a contradiction.  This proves that the partition is indeed essential in $G-v$.

We can now complete the proof of the claim.
Since $\X'=\X-\{X_0\}\cup\{X_0-v\}$ covers $E-\st_G(v)-F$, and no hinge of $\X_1$ or $\X_2$ contains $v$ (as $v$ is contained by $X_0$ only), 
we have 
\begin{multline*}
r_t(E_1-\st_G(v))+r_t(E_2) \\
\begin{aligned}
&\leq |F|+\sum_{X\in \X'}(3t|X|-6t)
-t \Bigl(\sum_{h\in H_1} (\deg_{\X_1}(h)-1) + \sum_{h\in H_2} (\deg_{\X_2}(h)-1)\Bigr) 
\\ &\leq |F|+\sum_{X\in \X}(3t|X|-6t)-3t- t \Bigl(\sum_{h\in H_1} (\deg_{\X_1}(h)-1) + \sum_{h\in H_2} (\deg_{\X_2}(h)-1)\Bigr) 
\\ &= r_t(E_1)+r_t(E_2)-3t<3t|V(G-v)|.
\end{aligned}
\end{multline*}
By the minimality of $|V|$ this implies that $G-v$ is not $12t$-connected. Since $|V| \geq 12t + 2$, this implies that $G-v$ has a set of vertices $U$ of cardinality 
at most $12t-1$ for which $G-v-U$ is disconnected. The graph $G-U$ must be connected by the $12t$-connectivity of $G$, which implies that 
each connected component of $G-v-U$ contains
at least one vertex which is a neighbor of $v$ in $G$. But the neighbors of $v$ in $G$ are all contained in $X$, and $G[X]$ is complete,
so there is an edge between each pair of connected components of $G-v-U$, a contradiction.
\end{proof}

If a vertex $v\in V$ satisfies
(i) (resp. (ii), (iii)) of Claim \ref{claim:types}, then we shall say that $v$ is of {\it type} (i) (resp. (ii), (iii)).
Let $H_i^v$ denote those hinges in $H_i$ which contain $v$, for $i=1,2$.

\begin{claim}
\label{claim2}
Each vertex $v\in V$ satisfies
\[
\label{hingeeq}
\frac{d_F(v)}{2t}+\sum_{X\in\X:v\in X} (3-\frac6{|X|})-\sum_{h\in H_1^v} \frac{\deg_{\X_1}(h)-1}2-\sum_{h\in H_2^v} \frac{\deg_{\X_2}(h)-1}2\geq 3.
\]
\end{claim}

\begin{proof}
We have three cases depending on the type of $v$.
Suppose that $v$ is of type (i). In this case at least two members of $\X$ contain $v$.
Since $\X_1$ and $\X_2$ are $4$-shellable, there is an ordering $(X^1_1,\dots, X^1_{q_1})$ of $\X_1$
and an ordering $(X^2_1,\dots, X^2_{q_2})$ of $\X_2$ for which
%
%
$|X^1_{i_1}\cap \bigcup_{j=1}^{i_1-1}X^1_j|\leq 4$ and $|X^2_{i_2}\cap \bigcup_{j=1}^{i_2-1}X^2_j|\leq 4$ hold for $2\leq i_1\leq q_1$ and $2\leq i_2\leq q_2$.
Let $(Y^1_1,\dots, Y^1_{p_1})$ and $(Y^2_1,\dots,Y^2_{p_2})$ be the inherited ordering of those sets in $\X_1$ and $\X_2$, respectively, 
that contain $v$.
We may have $p_i=0$ for some $i\in \{1,2\}$, but $p_1+p_2\geq 2$ follows from the fact that $v$ is of type (i).

We assign a non-negative integer $c_j^i$ to each set $Y^i_j$, $i=1,2$, $1\leq j\leq p_i$. 
If $Y^i_1$ exists then we put $c^i_1=0$, $i=1,2$, and we define $c^i_j=|Y^i_{j}\cap \bigcup_{j'=1}^{j-1}Y^i_{j'}|-1$, for $i = 1,2$ and $2\leq j\leq p_i$. 
Since the covers are $4$-shellable and $2$-thin, we have $c^i_j\leq 3$ and $c^i_2\leq 1$, whenever it is defined.
Furthermore, we have
 \[\sum_{j=1}^{p_i} c^i_j=\sum_{h\in H_i^v} ({\deg_{\X_i}(h)-1})\]
for $i=1,2$. We can now deduce that
\begin{equation}\label{count}
    \begin{aligned}[b]
        \frac{d_F(v)}{2t}+\sum_{X\in\X: v\in X} \bigl(3-\frac6{|X|}\bigr)-\sum_{h\in H_1^v} \frac{\deg_{\X_1}(h)-1}{2}-\sum_{h\in H_2^v}\frac{\deg_{\X_2}(h)-1}{2} \\ 
        \geq 0+\sum_{j=1}^{p_1}\bigl(3-\frac6{|Y_j^1|}-\frac{c_j^1}{2}\bigr)+\sum_{j=1}^{p_2}\bigl(3-\frac{6}{|Y_j^2|}-\frac{c_j^2}{2}\bigr).
    \end{aligned}
\end{equation}
Observe that there is at least one member of $\X$ that contains $v$ for which the corresponding $c$-value is $0$ (since either $Y_1^1$ or $Y_1^2$ must exist) and at least one other for which the corresponding $c$-value is at most $1$ (since $p_1+p_2\geq 2$). For these two sets, call them $X$ and $Y$, 
the sum of the corresponding terms in (\ref{count}) is at least 
 \[(3-\frac6{|X|})+(3-\frac6{|Y|}-\frac12)\geq 3-\frac65+3-\frac65-\frac12 = \frac{31}{10} >3.\]
For any other set $Z$, the corresponding $c$-value is at most three, and hence the corresponding term in (\ref{count}) is at least 
\[3-\frac6{|Z|}-\frac32\geq 3-\frac65-\frac32=\frac{3}{10}>0.\]
These inequalities imply the claim in the case when $v$ is of type (i).


Next suppose that
$v$ is of type (ii). Then $H_1^v=H_2^v=\varnothing$ and we have exactly one member of $X\in\X$ that contains $v$. As the degree of $v$ is at least $12t$ by 
the $12t$-connectivity of $G$, we also have  $d_F(v)+|X|\geq 12t+1$, where $d_F(v)\geq 1$ and $|X|\geq 5$ must hold. 
If $|X|\geq 12t$, then we have $\frac{d_F(v)}{2t}+(3-\frac{6}{|X|})\geq \frac{1}{2t}+(3-\frac{6}{12t})\geq 3$, as required. 
If $5\leq |X| \leq 12t$ then we have
\[\frac{d_F(v)}{2t}+(3-\frac{6}{|X|}) \geq \frac{12t+1-|X|}{2t}+(3-\frac{6}{|X|})\geq 3,\]where the second inequality follows from the fact that this is a quadratic inequality in $|X|$ that is satisfied at the two extremes $|X| = 5$ and $|X| = 12t$ of the domain, as can be verified by a simple computation. 

Finally, if
$v$ is of type (iii), then we have $\frac{d_F(v)}{2t}\geq \frac{12t}{2t}=6>3$, as claimed.
\end{proof}

By using Claim \ref{claim2} we can now deduce that
\begin{equation*}
\begin{aligned}
&r_t(E_1)+r_t(E_2) = 
\begin{aligned}[t]
&|F|+\sum_{X\in \X_1} (3t|X|-6t)-t\sum_{h\in H_1} (\deg_{\X_1}(h)-1) \\& +\sum_{X\in \X_2} (3t|X|-6t) -t\sum_{h\in H_2} (\deg_{\X_2}(h)-1)
\end{aligned}
\\ &= 
t\sum_{v\in V}\Big(\frac{d_F(v)}{2t} +\sum_{X\in\X:v\in X} \big(3-\frac{6}{|X|}\big)-\sum_{h\in H_1^v} \frac{\deg_{\X_1}(h)-1}{2} -\sum_{h\in H_2^v} \frac{\deg_{\X_2}(h)-1}{2}\Big)
\\ &\phantom{r_t(E_1)+r_t(E_2)} \geq 3t|V|, 
\end{aligned}
\end{equation*}
a contradiction. This completes the proof of the lemma.
\end{proof}

We have the following corollary, which extends \cite[Theorem 7.2]{CJT}.

\begin{corollary}
\label{clcofactrigid}
If $G=(V,E)$ is $12t$-connected for some positive integer $t$ then
$r_t(E)=3t|V|-6t$.
\end{corollary}
\begin{proof}
Suppose that $r_t(E)<3t|V|-6t$. Then $\{E,\varnothing\}$ is an essential partition and hence Lemma \ref{lemmaLYcofact} implies 
$r_t(E)=r_t(E)+r_t(\varnothing)\geq 3t|V|$, a contradiction.
\end{proof}

We can now deduce a strengthening of Lemma \ref{lemmaLYcofact}.

\begin{lemma}\label{lemmaLYcofact2}
Let $t$ be a positive integer and let $G=(V,E)$ be a $k$-connected graph with $k\geq 12t$.
Then for every essential partition
$\{E_1,E_2\}$ of $E$ we have 
\[
r_t(E_1)+r_t(E_2)\geq 3t|V|+k-12t.
\]
\end{lemma}

\begin{proof}
We prove the lemma by induction on $k$.
By Lemma \ref{lemmaLYcofact} the statement holds for $k=12t$, so we may assume that
$k\geq 12t+1$. Let $\{E_1,E_2\}$ be an essential partition of $E$.
Since $r_t(E)=3t|V|-6t$ by Corollary \ref{clcofactrigid}, we have $E_1,E_2\neq\varnothing$.
Hence there exists a vertex 
$v\in V(E_1)\cap V(E_2)$. 

Suppose first that $\{E_1 - \st_G(v), E_2 - \st_G(v)\}$ is not an essential partition of $G-v$. This means that either $r_t(E_1-\st_G(v))$ or $r_t(E_2-\st_G(v))$ is equal to $3t|V(G-v)|-6t$; by symmetry, we may assume that it is the former. It follows that $d_{E_1}(v)<3t$, for otherwise Lemma \ref{lemma:cofactorbridge} would imply that
$r_t(E_1)\geq r_t(E_1-\st_G(v)) + 3t = 3t|V|-6t$, contradicting our assumption that $\{E_1,E_2\}$ is essential. 
Observe that $r_t(\st_G(v))=d_G(v)$ by Lemma \ref{lemma:cofactorbridge}.
Using Lemma \ref{lemma:cofactorbridge} again we obtain
\[r_t(E_1)+r_t(E_2) \geq r_t(E_1-\st_G(v))+d_{E_1}(v)+d_{E_2}(v) \geq 3t|V(G-v)|-6t + k = 3t|V| -9t +k,\]
which is stronger than the bound we set out to prove.

Now let us consider the case when $\{E_1 - \st_G(v), E_2 - \st_G(v)\}$ is essential.
Then by the induction hypothesis we have \[r_t(E_1 - \st_G(v)) + r_t(E_2 - \st_G(v)) \geq 3t|V(G-v)| + k - 1 -12t.\]
Note that $d_G(v) \geq k \geq 3t+1$, and that by the choice of $v$, $d_{E_1}(v)$ and $d_{E_2}(v)$ are both non-zero. These observations together with Lemma \ref{lemma:cofactorbridge} imply that
\[r_t(E_1)+r_t(E_2)\geq r_t(E_1-\st_G(v))+r_t(E_2-\st_G(v))+3t+1\geq 3t|V|+k-12t,\]
as desired.
\end{proof}

We are ready to prove the main results of this section.

\begin{theorem}\label{LYcofact}
Let $G=(V,E)$ be a $k$-connected graph with $k\geq 12t$. Then 
$\cofactort(G)$ 
is vertically $(k-6t+1)$-connected.
\end{theorem}

\begin{proof}
Suppose that $\{E_1,E_2\}$ is  a vertical $c$-separation of $\cofactort(G)$ for some $c\leq k-6t$.
Then \[\max \{r_t(E_1), r_t(E_2)\} < r_t(E)=3t|V|-6t,\] where the equality follows from
Corollary \ref{clcofactrigid}.
Hence this partition of $E$ is essential.
Therefore, Lemma \ref{lemmaLYcofact2} yields $r_t(E_1)+r_t(E_2)\geq 3t|V|+k-12t\geq 3t|V|-6t+c$ which contradicts
the definition of a $c$-separation.
\end{proof}

Applying Theorem \ref{LYcofact} with $t=1$ and $k = 12$ gives that whenever $G$ is $12$-connected, $\mathcal{C}(G)$ is vertically $7$-connected, and in particular it is connected.
Thus we obtain an analogue of Theorem \ref{MLYl>k} for the $C_2^1$-cofactor matroid.

In one of our applications we shall also need the following result, which says that if $\cofactort(G)$ is sufficiently highly vertically connected, then $r_t(G) = 3t|V| -6t$. Note that for count matroids, a similar result follows immediately from Lemma \ref{bases}. However, the direct analogue of Lemma \ref{bases} for 
${\cal C}(G)$ is not true\footnote{To see this consider the so-called double banana graph $B$, obtained from two disjoint copies of $K_5$ by a 2-sum operation.
It is easy to see that ${\cal C}(B)$ is connected, but it is not even ${\cal C}$-rigid. Let $B'$ be obtained from $B$ by adding a new edge $e$ that connects two vertices 
of degree four. Then $B$ is induced by the edge set of a connected component of ${\cal C}(B+e)$, but it is not an induced subgraph of $B'$.}. 
Consequently, we have to work harder in this case.

\begin{theorem}\label{vconnrigid}
Let $G=(V,E)$ be a graph without isolated vertices.
Suppose that $\cofactort(G)$ 
is vertically $(6t+2)$-connected. Then $r_t(E)=3t|V|-6t$.
\end{theorem}

\begin{proof}
By Theorem \ref{theorem:cofactortrank},  there exists an edge set $F\subseteq E$ and a $4$-shellable $2$-thin cover $\X$ of $E-F$ with hinge set $H(\X)$ such that
\[r_t(E)=|F|+\sum_{X\in \X} (3t|X|-6t) - t \sum_{h\in H(\X)} (\deg_{\X}(h)-1).\] 
As noted before, the members of $F$ are bridges in $\cofactort(G)$. Since $\cofactort(G)$ is connected and thus bridgeless, we must have $F = \varnothing$.
It follows that if $|\X| = 1$, then $\X = \{V\}$ and thus $r_t(E) = 3t|V| - 6t$, as required. 

Let us assume, for a contradiction, that $|\X| \geq 2$. Consider the last set $X_q$ of a $4$-shellable ordering of the sets in $\X$. Since $|X_q\cap \bigcup(\X-\{X_q\}|\leq 4$,  at most $\binom{4}{2} = 6$ hinges are 
incident with $X_q$. Therefore \[t \sum_{h\in H(\X-\{X_q\})} (\deg_{\X-\{X_q\}}(h)-1) \geq t \sum_{h\in H(\X)} (\deg_{\X}(h)-1)-6t.\]
Thus we obtain, by considering $\mathcal{X} - X_q$ and using Theorem \ref{theorem:cofactortrank}, that
\begin{equation}\label{eq:X-Xqbound}
\begin{aligned}
r_t(E&-E(X_q)) \leq \sum_{X\in \X-\{X_q\}} (3t|X|-6t)-t \sum_{h\in H(\X-\{X_q\})} (\deg_{\X-\{X_q\}}(h)-1)
\\ &\leq \sum_{X\in \X} (3t|X|-6t)-t \sum_{h\in H(\X)} (\deg_{\X}(h)-1)+6t-(3t|X_q|-6t)
\\ &\leq r_t(E)+6t-(3t|X_q|-6t).
\end{aligned}
\end{equation}
On the other hand, $r_t(E(X_q))\leq 3t|X_q|-6t$. Hence $r_t(E-E(X_q))+r_t(E(X_q))\leq r_t(E) +6t$.
It follows that if $\min(r_t(E-E(X_q)),r_t(E(X_q)))\geq 6t+1$ then
$\{E-E(X_q),E(X_q)\}$ is a $(6t+1)$-separation 
of $\cofactort(G)$, contradicting the fact that  $\cofactort(G)$ 
is vertically $(6t+2)$-connected.
So in the rest of the proof we may assume that at least
one of $r_t(E-E(X_q))$ and $r_t(E(X_q))$ is at most $6t$.


Since $|\X|\geq 2$, there is a set $Y\in \X-X_q$. We have $|Y|\geq 5$ and $|X_q\cap Y|\leq 2$.
Thus $|X_q|\leq |V|-3$. Also, from Lemma \ref{lemma:cofactorbridge} and the fact that $\cofactort(G)$ is bridgeless we have $d_G(v)\geq 3t+1$ for each vertex $v \in V$. 
This implies, using Lemma \ref{lemma:cofactorbridge} again, that for any set $S\subseteq V-X_q$ with $|S|=3$ the edge set $\st_G(S)\subseteq E-E(X_q)$ contains
$9t-1$ edges 
which form an independent set in $\cofactort(G)$. Hence
$r_t(E-E(X_q))\geq 9t - 1\geq 6t+1$. Therefore we must have $r_t(E(X_q))\leq 6t$.

If $3t+1\leq r_t(E(X_q))\leq 6t$, then we have, by using the upper bound on $r_t(E-E(X_q))$ deduced earlier,
that \[r_t(E-E(X_q))+r_t(E(X_q))\leq r_t(E)+6t-(3t|X_q|-6t)+6t\leq r_t(E)+3t,\]
 which implies that $\{E-E(X_q),E(X_q)\}$ is a $(3t+1)$-separation in $\cofactort(G)$, a contradiction.
If $r_t(E(X_q))\leq 3t$, then we have \[r_t(E-E(X_q))+r_t(E(X_q))\leq r_t(E)+6t-(3t|X_q|-6t)+3t\leq r_t(E).\] Note that $E(X_q)\neq \varnothing$, since 
by (\ref{eq:X-Xqbound}) we have $r_t(E - E(X_q)) < r_t(E)$. In particular,
$r_t(E(X_q))\geq 1$. It follows that $\{E-E(X_q),E(X_q)\}$ is a $1$-separation
in $\cofactort(G)$, a contradiction.
This final contradiction completes the proof of the theorem.
\end{proof}

\section{Applications}\label{section:applications}

In this section we use the results of the previous sections to obtain analogues of Whitney's theorem for $(k,\ell)$-count matroids and for the $C_2^1$-cofactor matroid. Our proofs follow the basic strategy of the proof of Whitney's theorem due to Edmonds, see \cite[Lemma 5.3.2]{oxley}. We also pose a strengthening of a conjecture of Kriesell, and prove a special case of both of these conjectures.

\subsection{Recovering graphs from count matroids}

Let $G = (V,E)$ and $H = (V',E')$ be graphs and let $\psi: E \rightarrow E'$ be a function. We say that \emph{$\psi$ is induced by a graph isomorphism} if there is a graph isomorphism $\varphi: V \rightarrow V'$ such that for every edge $uv \in E,$ the endvertices of $\psi(uv)$ in $H$ are $\varphi(u)$ and $\varphi(v)$. 
The following lemma is well-known.

\begin{lemma}\label{lemma:star}
Let $G = (V,E), H = (V',E')$ be graphs without isolated vertices and let $\psi : E \rightarrow E'$ be a bijection that ``sends stars to stars'', that is, for every $v \in V$ there is a vertex $v' \in V'$ such that $\psi(\st_G(v)) = \st_H(v')$. Then $\psi$ is induced by a graph isomorphism.
\end{lemma}

For a matroid ${\cal M} = (E,r)$, we say that a subset $F \subseteq E$ is a \emph{$k$-hyperplane}, for some integer $k \geq 0$, if $r(F) = r(E) - k$ and $F$ is closed, i.e., $r(F + e) = r(F) + 1$ for every $e \in E - F$.

\begin{theorem} \label{theorem:countwhitney}
Let $k,\ell$ be integers with $k\geq 1$ and $\ell \leq 2k-1$, and let $c_{k,\ell} \geq 2$ be an integer such that every $c_{k,\ell}$-connected graph is $\countmatroid$-connected. Let $G= (V,E)$ and $H = (V',E')$ be graphs and $\psi: E \rightarrow E'$ an isomorphism between $\countmatroid(G)$ and $\countmatroid(H)$. If $G$ is $(c_{k,\ell}+1)$-connected and $H$ is without isolated vertices, then $\psi$ is induced by a graph isomorphism.
\end{theorem}
\begin{proof}
By Lemma \ref{bases} we have $r(E)=k|V|-\ell$, and similarly, $r(E')=k|V'|-\ell$. Hence $|V|=|V'|$.
Fix $v \in V$ and consider $F = E - \st_G(v)$. Since $F$ is the edge set of an induced subgraph of $G$, $F$ is closed in $\countmatroid(G)$. Moreover, since $G-v$ is $c_{k,\ell}$-connected, $\countmatroid(G-v)$ is connected, so by Lemma \ref{bases} we have $r(F) = k(|V|-1) -\ell = r(E) - k$. This shows that $F$ is a connected $k$-hyperplane in $\countmatroid(G)$, and consequently $\psi(F)$ is a connected $k$-hyperplane in $\countmatroid(H)$. By Lemma \ref{bases} again this means that
\[k|V(\psi(F))| - \ell = r(\psi(F)) = r(E') - k = r(E) - k = k(|V| - 1) - \ell,\]so $\psi(F)$ is the edge set of an induced subgraph of $H$ on $|V'|-1$ vertices, and hence it is the complement of a vertex star. This shows that $\psi$ maps complements of vertex stars to complements of vertex stars. Since $\psi$ is a bijection, it also follows that it maps vertex stars to vertex stars. Now Lemma \ref{lemma:star} implies that $\psi$ is induced by a graph isomorphism, as desired.
\end{proof}

In the case of the graphic matroid we can choose $c_{1,1} = 2$, and Theorem \ref{theorem:countwhitney} reduces to Whitney's theorem.
Similarly, by putting $c_{2,3} = 6$, we can reproduce  \cite[Theorem 2.4]{JK}. 
Combining Theorem \ref{theorem:countwhitney} with Corollaries \ref{MLYl<0}, \ref{MLYl<k} and \ref{MLYl>k}, we obtain the following generalization.

\begin{corollary}\label{corollary:countwhitney}
Let $k,\ell$ be integers with $k\geq 1$ and $\ell \leq 2k-1$. Let $G= (V,E)$ be a graph and suppose that one of the following holds:
\begin{enumerate}
    \item $\ell \leq 0$ and $G$ is $(c+1)$-connected, where $c = \lceil 2k - (2\ell-2)/|V|\rceil$; 
    \item $k \geq \ell > 0$ and $G$ is $(2k+1)$-connected;
    \item $2 \leq k < \ell$ and $G$ is $(2\ell+1)$-connected.
\end{enumerate}
Let $H = (V',E')$ be a graph without isolated vertices. If a function $\psi: E \rightarrow E'$ is an isomorphism between $\countmatroid(G)$ and $\countmatroid(H)$, then it is induced by a graph isomorphism.
\end{corollary}

We note that in the $\ell \leq 0$ case of Corollary \ref{corollary:countwhitney}, it would be enough to suppose that $G$ is $2$-connected and the degree of each vertex is at 
least $2k + 1 - (2\ell-2)/|V|$, provided that $G$ is simple. The proof, which we omit, is analogous to that of Theorem~\ref{theorem:countwhitney}.

For the bicircular matroid $\bicircular(G)$, Corollary \ref{corollary:countwhitney} implies that if $G$ is $4$-connected, then $\bicircular(G)$ uniquely determines $G$. Note that $2$-connectivity does not suffice: for example, a cycle and a path with the same number of edges are non-isomorphic graphs with isomorphic bicircular matroids. In fact, a result of Wagner \cite{wagner} implies that $3$-connectivity suffices. We give an alternative proof that is similar to our proof of Theorem \ref{theorem:countwhitney}. We shall rely on the following result.

\begin{lemma}\label{lemma:bicircularconnected} \cite[Proposition 2.4]{Matthews}
Let $G = (V,E)$ be a connected graph. Then $\bicircular(G)$ is connected if and only if $G$ is not a cycle and does not contain vertices of degree one.
\end{lemma}

\begin{theorem}
Let $G= (V,E)$ and $H = (V',E')$ be graphs and $\psi: E \rightarrow E'$ an isomorphism between $\bicircular(G)$ and $\bicircular(H)$. If $G$ is $3$-connected with at least five vertices and $H$ is without isolated vertices, then $\psi$ is induced by a graph isomorphism.
\end{theorem}
\begin{proof}
As in the proof of Theorem \ref{theorem:countwhitney}, we must have $|V| = |V'|$, and to each vertex $v \in V$ such that $\bicircular(G-v)$ is connected we can associate a vertex $v' \in V'$ such that $\psi(\st_G(v)) = \st_H(v')$. If $\bicircular(G-v)$ is connected for every $v \in V$, then we are done by Lemma \ref{lemma:star}. Thus we may assume that there is a vertex $v \in V$ for which  $\bicircular(G-v)$ is not connected. By Lemma \ref{lemma:bicircularconnected} and the assumption that $G$ is $3$-connected, $G-v$ must be a cycle. Therefore $G$ is isomorphic to the wheel graph $W_n$ for $n = |V| \geq 5$. It follows from Lemma \ref{lemma:bicircularconnected} that $\bicircular(G - u)$ is connected for every vertex $u \in (V - v)$. As before, this implies that $\psi$ sends $\st_G(u)$ to some vertex star $\st_H(\varphi(u))$ for all $u \in (V-v)$. The function $\varphi : (V-v) \rightarrow V' $ defined in this way is injective, so there is a unique vertex $v' \in (V' - \text{im}(\varphi))$ and we have $\psi(\st_G(v)) = \st_H(v')$. Now Lemma \ref{lemma:star} implies, again, that $\psi$ is induced by a graph isomorphism.
\end{proof}

In the $k < \ell$ case the bound given by Corollary \ref{corollary:countwhitney}(c) is almost tight, in the sense that there exist $(2\ell-1)$-connected graphs that are not uniquely determined by their count matroids. Indeed, let $G$ be a $(2\ell-1)$-connected graph that is not $\countmatroid$-connected (we gave a construction for such a graph in Section \ref{section:countmatroids}), and let $(E_1,E_2)$ be a vertical $1$-separation of $G$. Consider the graph $H$ obtained as the disjoint union of $G[E_1]$ and $G[E_2]$. Then $\countmatroid(G)$ and $\countmatroid(H)$ are isomorphic under the natural edge bijection between $G$ and $H$, but the graphs are nonisomorphic. Similar examples can be constructed for the cases when $\ell \leq k$; we omit the details.

\subsection{Recovering graphs from cofactor matroids}

We can also prove an analogue of Whitney's theorem for the $t$-fold union of the $C_2^1$-cofactor matroid.
Again, we stress that the analogue of Lemma \ref{bases} does not hold for this matroid:
in particular, the connectivity of $\cofactor(G)$ does not imply the ``$\mathcal{C}$-rigidity" of $G$. To overcome this difficulty,
we shall use our results on (unions of) cofactor matroids with high vertical connectivity.

\begin{theorem} \label{theorem:whitneycof}
Let $t$ be a positive integer and
let $G= (V,E),H = (V',E')$ be graphs and $\psi: E \rightarrow E'$ an isomorphism between $\cofactort(G)$ and $\cofactort(H)$. If $G$ is $(12t+2)$-connected and $H$ is without isolated vertices, then $\psi$ is induced by a graph isomorphism.
\end{theorem}
\begin{proof}
Fix $v \in V$ and consider $F = E - \st_G(v)$. Let $H_0$ denote the subgraph of $H$ induced by $\psi(F)$; note that $\cofactort(G-v)$ and $\cofactort(H_0)$ are isomorphic. Since $G$ is $(12t+2)$-connected and $G-v$ is $(12t+1)$-connected, $\cofactort(G)$ is vertically $(6t+3)$-connected and $\cofactort(G-v)$ is vertically $(6t+2)$-connected by Theorem \ref{LYcofact}, which imply $r_t(E) = 3t|V| - 6t$ and $r_t(F) = 3t(|V|-1) - 6t= r_t(E) - 3t$ by Theorem \ref{vconnrigid}. 
It follows that \[3t|V(H_0)| - 6t = r_t(\psi(F)) = r_t(F) = r_t(E) - 3t = r_t(E') - 3t = 3t|V'| - 6t - 3t,\]where the first and last equalities follow from Theorem \ref{vconnrigid}, applied to $\cofactort(H_0)$ and to $\cofactort(H)$, respectively. Thus, we have $|V(H_0)| = |V'| - 1$.
Also, since $F$ is the edge set of an induced subgraph of $G$, it is closed in $\cofactort(G)$, and thus $\psi(F)$ is closed in $\cofactort(H)$. It follows that $H_0$ is an induced subgraph of $H$, since $r_t(\psi(F)) = 3t|V(H_0)| - 6t$ implies that adding edges induced by $V(H_0)$ to $\psi(F)$ cannot increase the rank of the latter. 

To summarize, $\psi(F)$ is the edge set of an induced subgraph of $H$ on $|V'| - 1$ vertices, which implies that it is the complement of a vertex star. This shows that $\psi$ maps complements of vertex stars to complements of vertex stars. Since $\psi$ is a bijection, it also follows that it maps vertex stars to vertex stars. Now Lemma \ref{lemma:star} implies that $\psi$ is induced by a graph isomorphism, as desired.
\end{proof}

Again, the bound given by Theorem \ref{theorem:whitneycof} is not far from being tight, at least for $t=1$: there are examples of $11$-connected graphs that are not determined by their $C_2^1$-cofactor matroids. As in the case of count matroids, we can take an $11$-connected graph $G$ for which $\cofactor(G)$ is not connected. An example of such a graph was constructed by Lovász and Yemini \cite{LY}. Now consider a $1$-separation $(E_1,E_2)$ of $\cofactor(G)$ and observe that the disjoint union of $G[E_1]$ and $G[E_2]$ has the same cofactor matroid as $G$, but is not isomorphic to $G$.

In order to deduce a Whitney-type result in terms of the vertical connectivity of $\cofactort(G)$, rather than the connectivity of $G$, we
need the converse of Theorem \ref{LYcofact}. Since the proof of the corresponding result for count matroids (the converse of
Theorem \ref{theorem:countwhitney}) is very similar,
we prove the two statements simultaneously.

\begin{theorem}\label{theorem:vconnconnmix}
Let $c\geq 5$ be an integer and let $G=(V,E)$ be a graph on at least $c+2$ vertices, without isolated vertices. 
Let $\cM=(E,r)$ be either $\countmatroid(G)$, for some integers $k\geq 1$ and $\ell \leq 2k-1$, or $\cofactort(G)$, for some integer $t\geq 1$. 
In the latter case put
$k=3t$ and $\ell=6t$. 
If $\cM$
is vertically $(k(c-1)-\ell+2)$-connected, then $G$ is $c$-connected.
\end{theorem}

\begin{proof}
We have $r(E)=k|V|-\ell$ by Lemma \ref{bases} or Theorem \ref{vconnrigid}.
For a contradiction suppose that there is a set
$S\subset V$ with $|S| \leq k-1$ for which
$G-S$ is disconnected. By adding vertices to $S$ we may suppose that $|S| = c-1$. Let $C$ be the vertex set of some 
connected component of $G-S$.
We define a partition $\{E_1,E_2\}$ of $E$ by letting
$E_1=E(G[C])\cup \st_G(C)$ and $E_2=E(G[V-C])$.

Since $E_1\subseteq G[C\cup S]$ and $|C\cup S|, |V-C| \geq c\geq 3$, we have
$r(E_1)\leq k|C\cup S|-\ell$ and $r(E_2) \leq k|V-C|-\ell$. Hence 
\begin{equation}
\label{boundmix}
r(E_1)+r(E_2)\leq k|V|-\ell+k|S|-\ell = r(E) + k|S|-\ell.
\end{equation}

Let $w=\min \{ r(E_1), r(E_2) \}$.
First suppose that $w\geq k|S|-\ell+1$.
Then $\{E_1,E_2\}$ is a $(k|S|-\ell+1)$-separation by
\eqref{boundmix}. Since
$k|S|-\ell+1 = k(c-1)-\ell+1$, this contradicts the assumption
on the vertical connectivity of $\cM$.

Thus we may suppose that there is an integer $z$ with $1\leq z\leq |S|$ for which
\[
k(|S|-z+1)-\ell\geq w \geq k(|S|-z)-\ell+1
\]
holds. By definition, we have $w= r(E_1)$ or $w= r(E_2)$. In the
latter case we obtain
\begin{eqnarray*}
 r(E_1)+r(E_2)&\leq& k|C\cup S|-\ell + k(|S|-z+1)-\ell \\ &=&
k|V|-\ell+k(|S|-z+1)-\ell-k|V-C-S|\leq r(E) + k(|S|-z)-\ell,
\end{eqnarray*}
using that $|V-C-S|\geq 1$.
Therefore $\{E_1,E_2\}$ is a vertical $(k(c-z)-\ell+1)$-separation,
contradicting our assumption on the vertical connectivity of $\cM$.
In the former case a similar count (using $r(E_2)\leq k|V-C|-\ell$
and $|C|\geq 1$) leads to the same contradiction.
\end{proof}

In the proof of Theorem \ref{theorem:vconnconnmix} the assumption $c\geq 5$ is required
for the cofactor matroid result only (it is used in the first line of the
proof, where we apply Theorem \ref{vconnrigid}). For the count matroid part the bound $c\geq 3$ would
suffice.
 
Combining the $t=1$ cases of Theorem \ref{theorem:whitneycof}  
and Theorem \ref{theorem:vconnconnmix}
yields the following answer to the cofactor version of the
question of Brigitte and Herman Servatius that we mentioned in Section \ref{intro}.

\begin{theorem}\label{theorem:servatius}
Let $G= (V,E),H = (V',E')$ be graphs without isolated vertices and let $\psi: E \rightarrow E'$ be an isomorphism between $\cofactor(G)$ and $\cofactor(H)$. 
If $\cofactor(G)$ is vertically $35$-connected, 
then $\psi$ is induced by a graph isomorphism.
\end{theorem}

By combining 
Theorems \ref{theorem:countwhitney} and \ref{theorem:vconnconnmix}, we can deduce 
analogous results for the count matroid $\countmatroid(G)$, for all $k\geq 1$ and $\ell \leq 2k-1$. 
A result of this type was obtained earlier in
\cite{JK} in the case when $(k,\ell) = (2,3)$.

\subsection{Removable spanning trees}

The following conjecture is due to Matthias Kriesell (see \cite[Problem 444]{kriesell}).

\begin{conjecture}\label{conjecture:Kriesell}
For every positive integer $k$ there exists a (smallest) integer $f(k)$ such that every $f(k)$-connected
graph $G$ contains a spanning tree $T$ for which $G-E(T)$ is $k$-connected.
\end{conjecture}

It follows from Theorem \ref{theorem:tutte} that $f(1)=4$. 
The bound $f(2) \leq 12$ was obtained in
\cite{J2conn}, which was subsequently improved to $f(2) \leq 8$ in \cite{cds}.
To the best of our knowledge, the conjecture is still open for $k\geq 3$.
In fact, in both the $k=1$ and $k=2$ cases a much stronger statement is true, which we put forward as a conjecture for general $k$.
\begin{conjecture}\label{conjecture:extended}
For every positive integer $k$ there exists a (smallest) integer $f^*(k)$ such that for every positive integer $t$, every $(t \cdot f^*(k))$-connected graph contains $t$ edge-disjoint $k$-connected spanning subgraphs.
\end{conjecture}

With this notation, we have $f^*(1) = 2$, while \cite[Theorem 1.1]{J2conn} asserts that $f^*(2) \leq 6$. By using our results on cofactor matroids we now settle the $k=3$ case of Conjecture \ref{conjecture:extended}, and thus of Conjecture \ref{conjecture:Kriesell}. The proof is analogous to the proof of the $k=2$ case in \cite{J2conn}.
\begin{theorem}\label{theorem:kriesell}
Every $12t$-connected graph contains $t$ edge-disjoint $3$-connected spanning subgraphs. In other words, $f^*(3) \leq 12$. It follows that $f(3) \leq 24$.
\end{theorem}
\begin{proof}
Let $G = (V,E)$ be a $12t$-connected graph. Corollary \ref{clcofactrigid} implies that $r_t(E) = 3t|V| - 6t$. This means that $G$ contains $t$ edge-disjoint subgraphs $H_1,\ldots,H_t$ with $r_1(E(H_i)) = 3|V| - 6$ for $i= 1,\ldots,t$. Since $r_1(E(H_i)) \leq 3|V(H_i)| - 6$ also holds, we must have $V = V(H_i)$, that is, $H_1,\ldots,H_t$ are spanning subgraphs. Finally, it is known 
that $r_1(E(H_i)) = 3|V(H_i)| - 6$ implies that $H_i$ is $3$-connected (it follows from \cite[Lemma 10.2.4]{Whlong}, or it 
can also be verified using Theorem \ref{cofactorrank}).
\end{proof}

Extrapolating from the above bounds on $f^*(k)$ for $k \leq 3$, perhaps  $f^*(k) \leq k(k+1)$ holds. A natural strategy to prove Conjecture \ref{conjecture:extended} is to prove a ``basis packing result'' similar to Theorem \ref{clcofactrigid} for some matroid $\mathcal{M}(K_n)$ on the edge set of the complete graph $K_n$ in which spanning sets are $k$-connected.
The family of \emph{abstract $k$-rigidity matroids} has this property. However, for $k \geq 4$ there is no known example of an abstract $k$-rigidity matroid where a good characterization of the rank function is available. In this context we wish to highlight \cite[Conjecture 8.1]{CJT}, which, if true, would give such an example. It is conceivable that a proof of this conjecture, combined with the methods of \cite{CJT} and this paper, would lead to a proof of Conjecture \ref{conjecture:extended}. 

\section*{Acknowledgements}

This work was supported by the Hungarian Scientific Research Fund grant Nos. FK128673, K135421, and PD138102.
The first author was supported by the \'UNKP-22-3 New National Excellence Program of the
Ministry for Innovation and Technology. The second author was supported in part by the MTA-ELTE Momentum Matroid Optimization Research Group and the
National Research, Development and Innovation Fund of Hungary, financed under the ELTE
TKP 2021‐NKTA‐62 funding scheme.
The third author was supported by the J\'anos Bolyai Research Scholarship of the Hungarian Academy of Sciences and by the \'UNKP-21-5 New National Excellence Program of the Ministry for Innovation and Technology. 

\section*{Appendix}

\noindent{\bf Lemma A.}
{\em
Let $k,\ell$ be integers with $k \geq 1$ and $\ell \leq 2k-1$. Let $G = (V,E)$ be a graph and
suppose that \[m=\min \{ \sum_{Y \in {\cal Y}} (k|V(Y)|-\ell) : {\cal Y}\ \hbox{is a partition of}\ E \}.\]
Then
\begin{enumerate}
    \item there exists a $1$-thin cover $\mathcal{X}$ of $G$ for which
$\val_{k,\ell}({\cal X})=m$ holds. Furthermore,
    \item if $0< \ell \leq k$, then there exists a $0$-thin cover ${\cal X}$ of $G$ with $\val_{k,\ell}({\cal X})=m$, and
    \item if $\ell \leq 0$, then we have $m=k|V(E)|-\ell$, that is, $\mathcal{X} = \{V(E)\}$ is a minimizing cover.
\end{enumerate}
}

\begin{proof}
Let $\mathcal{Y}$ be a partition of $E$ and suppose that there are members $Y_i,Y_j \in \mathcal{Y}$ with $|V(Y_i) \cap V(Y_j)| \geq 2$. Then we have 
$$k|V(Y_i)|-\ell+ k|V(Y_j)|-\ell = k|V(Y_i)\cup V(Y_j)|-\ell + k|V(Y_i)\cap V(Y_j)|-\ell =$$
$$k|V(Y_i\cup Y_j)|-\ell + k|V(Y_i)\cap V(Y_j)|-\ell
>  k|V(Y_i\cup Y_j)|-\ell,$$
where the last inequality follows from $\ell < 2k$. Thus for $\mathcal{Y}' = \mathcal{Y} - Y_i - Y_j + (Y_i \cup Y_j)$ we have 
\[\sum_{Y' \in \mathcal{Y}'}(k|V(Y')| - \ell < \sum_{Y \in {\cal Y}} (k|V(Y)|-\ell).\] It follows that if $\mathcal{Y}$ is a minimizing partition, then $|V(Y_i) \cap V(Y_j)| \leq 1$ holds for every pair of members $Y_i,Y_j \in \mathcal{Y}$, and thus $\mathcal{X} = \{V(Y): Y \in \mathcal{Y}\}$ is a $1$-thin cover of $G$ with $m = \val_{k,\ell}(\mathcal{X})$.

Parts \textit{(b)} and \textit{(c)} can be deduced similarly by observing that in the $\ell \leq k$ case, we have
\[k|V(Y_i\cup Y_j)|-\ell + k|V(Y_i)\cap V(Y_j)|-\ell \geq k|V(Y_i\cup Y_j)|-\ell\] for $Y_i,Y_j \in \mathcal{Y}$ whenever $|V(Y_i)\cap V(Y_j)|\geq 1$, and in the $\ell \leq 0$ case the same inequality holds for all pairs $Y_i,Y_j \in \mathcal{Y}$.
\end{proof}

\end{document}